\documentclass[11pt]{amsart}
\usepackage{amssymb,amsmath,amsthm,epsfig}
\theoremstyle{plain}

\newtheorem{thm}{Theorem}[section]
\newtheorem{lem}[thm]{Lemma}
\newtheorem{cor}[thm]{Corollary}
\newtheorem{prop}[thm]{Proposition}
\theoremstyle{definition}
\newtheorem{defin}[thm]{Definition}

\def\hangbox to #1 #2{\vskip1pt\hangindent #1\noindent \hbox to #1{#2}$\!\!$}

\allowdisplaybreaks

\title{Multivariate polynomial interpolation and sampling in Paley-Wiener
spaces}

\author{B. A. Bailey}
\address{Department of Mathematics, Texas A\&M University\\  
College Station, TX 77843, USA}
\email{abailey@math.tamu.edu}
\thanks{This research was supported in part by the NSF Grant DMS0856148}

\begin{document}

\begin{abstract}
In this paper, an equivalence between existence of particular exponential Riesz bases for multivariate bandlimited functions and existence of certain polynomial interpolants for these bandlimited functions is given.  For certain classes of unequally spaced data nodes and corresponding $\ell_2$ data, the existence of these polynomial interpolants allows for a simple recovery formula for multivariate bandlimited functions which demonstrates $L_2$ and uniform convergence on $\mathbb{R}^d$.  A simpler computational version of this recovery formula is also given, at the cost of replacing $L_2$ and uniform convergence on $\mathbb{R}^d$ with $L_2$ and uniform convergence on increasingly large subsets of $\mathbb{R}^d$. As a special case, the polynomial interpolants of given $\ell_2$ data converge in the same fashion to the multivariate bandlimited interpolant of that same data.  Concrete examples of pertinant Riesz bases and unequally spaced data nodes are also given.
\end{abstract}

\maketitle

\section{Introduction}\label{S:0}
\noindent  Approximation of bandlimited functions as limits of polynomials has a long history, as the following question illustrates: if $(\mathrm{sinc}\pi(\cdot -t_n))_{n \in \mathbb{Z}}$ forms a Riesz basis for $PW_{[-\pi,\pi]}$, what are the canonical product expansions of the biorthogonal functions for this Riesz basis? The first results along these lines were given by Paley and Wiener in \cite{PW}, and improved upon by Levinson in \cite[pages 47-67]{L}), while Levin extends these results to different classes of Riesz bases in \cite{Lev}.  A complete solution is given by Lyubarskii and Seip in \cite{LS} and Pavlov in \cite{P}.  In particular, they prove the following theorem:
\begin{thm}\label{lyuseippav}
Let $(t_n)_n \subset \mathbb{R}$, where $t_n \neq 0$ when $n \neq 0$, be a sequence such that the family of functions $(\mathrm{sinc} \pi(\cdot -t_n))_n$ is a Riesz basis for $PW_{[-\pi,\pi]}$, then the function
\begin{equation}
S(z) = \lim_{r \rightarrow \infty} (z-t_0) \prod_{\{ t_n \ : \ |t_n|<r, n \neq 0 \}} \Big(1-\frac{z}{t_n}  \Big)\nonumber
\end{equation}
is entire, where convergence is uniform on compacta,  and the biorthogonal functions $(G_n)_n$ of $$(\mathrm{sinc} \pi((\cdot) -t_n))_n$$ are given by
\begin{equation}
G_n(z) = \frac{S(z)}{(z-t_n)S'(t_n)}.\nonumber
\end{equation}
\end{thm}

\noindent The following is a readily proven corollary of Theorem \ref{lyuseippav}:
\begin{cor}\label{lyuseippavcor}
Let $(t_n)_n\subset \mathbb{R}$ and $(G_k)_k$ be defined as in Theorem \ref{lyuseippav}, then for each $k$, there exists a sequence of polynomials $(\Phi_{N,k})_N$ such that\\

\noindent 1) $\Phi_{N,k} (t_n) = G_k (t_n)$ when $|t_n|<N$.\\

\noindent 2) $\lim_{N \rightarrow \infty} \Phi_{N,k} = G_k$ uniformly on compacta.

\end{cor}

\noindent Corollary \ref{lyuseippavcor} raises two questions:\\

\noindent 1) Does every multivariate bandlimited function, (not just biorthogonal functions associated with a particular exponential Riesz basis), have a corresponding sequence of polynomial interpolants?\\

\noindent 2) If such polynomial interpolants for a multivariate bandlimited function exist, can these interpolants be used to be approximate the function in some simple and straightforward way?\\

\noindent  Let $(t_n)_{n \in \mathbb{Z}^d} \subset \mathbb{R}^d$ be a sequence such that the family of exponentials $\big( e^{i \langle \cdot , t_n \rangle} \big)_{n \in \mathbb{Z}^d}$ is a uniformly invertible Riesz basis for $L_2([-\pi,\pi]^d)$ (defined in section 4).  Under this condition, Theorem \ref{maintheorem} answers the first question affirmatively by showing that multivariate bandlimited functions can be approximated globally, both uniformly and in $L_2$, by a rational function times a multivariate sinc function.  Stated informally,
\begin{equation}\label{firstblip}
f(t) \simeq  \Psi_\ell(t) \frac{\mathrm{SINC} (\pi t)}{Q_{d,\ell} (t)},\quad \ell > 0,
\end{equation}
where $(\Psi_\ell))_{\ell \in \mathbb{N}}$ is a particular sequence of interpolating polynomials and $(Q_{d,\ell})_{d,\ell}$ is a sequence of polynomials which removes the zeros of the SINC function.  This gives a partial answer to the second question, but the fraction in the approximants above becomes more complex as $\ell$ increases.  Theorem \ref{simplethm} gives a more satisfactory answer to question 2) by using $$e^{-\sum_{k=1}^{N} \frac{1}{k(2k-1)}\frac{\| t \|_{2k}^{2k}}{(\ell+1/2)^{2k-1}}}, \quad \ell>0$$ in lieu of the SINC function in expression (\ref{firstblip}).  The exponent in the above expression is simply a rational function of $\ell$.  This simplicity necessitates replacing global $L_2$ and uniform convergence with a more local (though not totally local) convergence.  Corollary \ref{simplecormod} is of particular interest as a multivariate analogue of Theorem \ref{lyuseippav}, stately informally as
\begin{equation}
f(t) \simeq  \Psi_\ell(t),\quad \ell > 0.\nonumber
\end{equation}
The author is unaware of any other multivariate polynomial approximation theorem which applies to exponential Riesz bases which are not necessarily tensor products of single-variable Riesz bases, or that demonstrate convergence stronger than uniform convergence on compacta.  As a note, Theorems \ref{maintheorem}, \ref{simplethm}, and Corollary \ref{simplecormod} do not, at this point, recover Corollary \ref{lyuseippavcor} in its generality of allowable sequences $(t_n)_n\subset \mathbb{R}$; however, the comments above show that their value is primarily due to their multidimensional nature and convergence properties.\\

\noindent This paper is outlined as follows.  Section 2 covers the necessary preliminary and background material regarding bandlimited functions, and section 3 outlines some basic properties of uniformly invertible operators.  Theorems \ref{maintheorem} and Theorem \ref{simplethm} are proven in sections 4 and 5 respectively, along with pertinant corollaries.  Section 6 gives explicit examples of sequences $(t_n)_{n \in \mathbb{Z}^d}$ to which Theorems \ref{maintheorem} and \ref{simplethm} apply.  Section 7 (as an appendix) addresses the optimality of growth rates appearing in Theorem \ref{simplethm}.

\section{Preliminaries}\label{S:1}

\begin{defin}
A reproducing kernel Hilbert space is a Hilbert space $H$ of functions on $X$ such that there exists $K:X\times X\rightarrow \mathbb{C}$ satisfying the following:\\
\noindent 1) For all $y \in X$, $K(\cdot,y) \in H$.\\
\noindent 2) $f(x) = \langle f(\cdot) , K(\cdot,x) \rangle$ for all $x \in X$ and $f\in H.$
\end{defin}

\begin{defin}
A Riesz basis for a Hilbert space $H$ is a sequence $(f_n)_{n \in \mathbb{N}}$ which is isomorphically equivalent to an orthonormal basis of $H$.  Equivalently, a Riesz basis is an unconditional Schauder basis.\\
\end{defin}

\noindent If $(f_n)_{n \in \mathbb{N}}$ is a Schauder (Riesz) basis for a Hilbert space $H$, then there exists a unique set of functions $(f_n^*)_{n \in \mathbb{N}}$, (the biorthogonals of $(f_n)_{n \in \mathbb{N}}$) such that $\langle f_n , f_m^* \rangle  = \delta_{nm}.$ The biorthogonals also form a Schauder (Riesz) basis for $H$. Note that biorthogonality is preserved under a unitary transformation.\\

\noindent We use the $d$-dimensional $L_2$ isometric Fourier transform
\begin{equation}
\mathcal{F}(f)(\cdotp) = \mathrm{P.V. \ }\frac{1}{(2\pi)^{d/2}} \int_{\mathbb{R}^d} f(\xi)e^{-i \langle \cdotp , \xi \rangle} d\xi, \quad f  \in L_2(\mathbb{R}^d),\nonumber
\end{equation}
where the inverse transform is given by 
\begin{equation}
\mathcal{F}^{-1}(f)(\cdotp) = \mathrm{P.V. \ } \frac{1}{(2\pi)^{d/2}}\int_{\mathbb{R}^d} f(\xi)e^{i \langle \cdotp , \xi \rangle} d\xi , \quad f  \in L_2(\mathbb{R}^d).\nonumber
\end{equation}

\begin{defin}
We define $PW_{[-\pi,\pi]^d} := \{f \in L_2(\mathbb{R}^d) \arrowvert \mathrm{supp}(\mathcal{F}^{-1}(f)) \subset [-\pi,\pi]^d \}$, with the inherited $L_2 (\mathbb{R}^d)$ norm.  Functions in $PW_{[-\pi,\pi]^d}$ are also called bandlimited functions.
\end{defin}

\noindent Here are facts concerning $PW_{[-\pi,\pi]^d}$ which will be used ubiquitously.\\

\noindent 1) $PW_{[-\pi,\pi]^d}$ is isometric to $L_2([-\pi,\pi]^d)$ by way of the Fourier transform.\\

\noindent 2) $PW_{[-\pi,\pi]^d}$ consists of entire functions, though in this paper we restrict the domain to $\mathbb{R}^d$.\\

\noindent 3) $PW_{[-\pi,\pi]^d}$ is a reproducing kernel Hilbert space with reproducing kernel $$K(x,y) = \mathrm{SINC}\pi(x-y)$$ where $$\mathrm{SINC}(x) := \mathrm{sinc}(x(1))\cdot\ldots\cdot\mathrm{sinc}(x(d)),\quad \mathrm{sinc}(x) := \frac{\sin(x)}{x}.$$

\noindent 4) $\big(\mathrm{SINC}\pi((\cdot)-n)\big)_{n \in \mathbb{Z}^d}$ is an orthonormal basis for $PW_{[-\pi,\pi]^d}$.  This follows from $$\mathcal{F}\big( \frac{1}{\sqrt{2\pi}}e^{i \tau (\cdot)} \big)(t) = \mathrm{sinc}\pi(t-\tau).$$

\noindent 5) \noindent In $PW_{[-\pi,\pi]^d}$, $L_2$ convergence implies uniform convergence.\\

\noindent 6) If $f\in PW_{[-\pi,\pi]^d}$, then $$\lim_{\|x\|_\infty \rightarrow \infty} f(x)=0.$$  This follows from the $d$-dimensional Riemann-Lebesgue Lemma.\\

\noindent 7)  The following result \cite[Theorem 19.3]{W} due to Paley and Wiener characterizes single-variable bandlimited functions.
\begin{thm}\label{PaleyWiener}
A function $f$ is in $PW_{[-\pi,\pi]}$ if and only if the following statements hold.\\
1) $f$ is entire.\\
2) There exists $M \geq 0$ such that $|f(z)| \leq M e^{\pi |z|}$ for $z \in \mathbb{C}$.\\
3) $f\big|_\mathbb{R} \in L_2(\mathbb{R})$.
\end{thm}

\section{Uniform invertibility of operators and Riesz bases}\label{S:2}

\begin{defin} Let $A:\ell_2(\mathbb{N}) \rightarrow \ell_2(\mathbb{N})$ be an onto isomorphism. Regard $A$ as a matrix map with respect to the unit vector basis of $\ell_2(\mathbb{N})$.  Let $\pi_k$ be the orthogonal projection onto the span of the first $k$ terms of the unit vector basis.  If 
\begin{equation}\label{matunif}
\sup_{j \in \mathbb{N}} \| (\pi_{k_j} A \pi_{k_j})^{-1}\|< \infty,
\end{equation}
for an increasing sequence $(k_j)_{j \in \mathbb{N}},$ then $A$ is said to be uniformly invertible as a matrix map with respect to the projections $(\pi_{k_j})_{j \in \mathbb{N}}$.  The terms in inequality (\ref{matunif}) should be interpreted as standard matrix norms and inverses of finite dimensional matrices.
\end{defin}

\noindent Let $S$ be an orthonormal basis for a Hilbert space $H$. Let $(S_n)_{n \in \mathbb{N}}$ be a sequence of sets such that\\

\noindent 1) $\emptyset \neq S_1 \subsetneq S_2 \subsetneq \cdots \subset S$, and\\
\noindent 2) $\bigcup_{n=1}^{\infty} S_n = S$.\\

\noindent Define $P_\ell$ to be the orthogonal projection onto $\mathrm{span}\{e_k\}_{e_k \in S_\ell}$.  Note that
\begin{equation}
\lim_{\ell \rightarrow \infty} P_\ell x = x, \quad x \in H.
\end{equation}
Linearly order $S = (e_n)_{n \in \mathbb{N}}$ such that, if $e_n \in S_k \setminus S_{k-1}$, and $e_m \in S_{k-1}$, then $m < n$.\\

\begin{defin}
Let $(v_k)_{k \in \mathbb{N}}$ be the unit vector basis for $\ell_2(\mathbb{N})$ and define $\phi$ by $\phi e_k = v_k$.  Let $L:H \rightarrow H$ be an onto isomorphism. $L$ is said to be uniformly invertible with respect to the projections $(P_\ell)_{\ell \in \mathbb{N}}$ if $\phi L \phi^{-1}$ is uniformly invertible as a matrix map with respect to the projections $(\pi_{|S_\ell|})$.
\end{defin}

\noindent We define the following notation:
\begin{equation}\label{abusivenot}
(P_\ell L P_\ell)^{-1} := (\pi_{|s_\ell|} (\phi L \phi^{-1}) \pi_{|s_\ell|})^{-1}.
\end{equation}
By saying $P_\ell L P_\ell$ is invertible, we mean that the right hand side of equation (\ref{abusivenot}) is well defined.  If $L$ is defined on $\mathrm{span}(e_n)_{n \in \mathbb{N}}$, but perhaps not on $H$, we define ``$P_\ell L P_\ell$ is invertible'' in the same way.\\

\noindent If $L$ is an operator on $H$ (perhaps densely defined), and $(P_\ell)_{\ell \in \mathbb{N}}$ is a sequence of projections defined above, define the operator $L_\ell = L P_\ell + I-P_\ell.$\\

\begin{defin}
Let $(v_k)_{k \in \mathbb{N}}$ be a Riesz basis for $H$.  We define $(v_k)_{k \in \mathbb{N}}$ to be uniformly invertible with respect to the projections $(P_\ell)_{\ell \in \mathbb{N}}$ if the corresponding isomorphism $L e_k = v_k$ is uniformly invertible with respect to the projections $(P_\ell)_{\ell \in \mathbb{N}}$.
\end{defin}

\noindent We can now state and prove the following lemmas:

\begin{lem}\label{firstone}
Let $(e_n)_{n \in \mathbb{N}}$ be an orthonormal basis for $H$,  let $(f_n)_{n \in \mathbb{N}} \subset H$, and let $P_\ell$ be the orthogonal projection onto $\mathrm{span}(e_n)_{n \leq \ell}$.  Define $L: \mathrm{span}\{e_n\}_{n \in \mathbb{N}} \rightarrow H$ by $Le_n= f_n$.  For each $\ell>0$, the following statements are equivalent:\\

\noindent 1) $(f_n)_{n \leq \ell} \cup (e_n)_{n > \ell}$ is a Riesz basis for $H$.\\
\noindent 2) $P_\ell L P_\ell$ is invertible.
\end{lem}
\begin{proof}[Proof of Lemma \ref{firstone}]
\noindent 1) $\Longrightarrow$ 2): From the definition of $L_\ell$ we know that it is an onto isomorphism. This yields $P_\ell = P_\ell L P_\ell L_\ell^{-1}$, implying $P_\ell = (P_\ell L P_\ell)(P_\ell L_\ell^{-1} P_\ell)$.\\ 

\noindent 2) $\Longrightarrow$ 1):  Let $A_\ell$ be the unique square matrix such that $P_\ell L P_\ell A_\ell = A_\ell P_\ell L P_\ell = P_\ell$. We need to show that $L_\ell$ is an onto isomorphism.\\

\noindent First we show that $L_\ell$ is one to one.  Say $ 0 = L_\ell x = L P_\ell x +(I-P_\ell)x$, then $0= P_\ell L P_\ell x$, so that $0 = A_\ell P_\ell L P_\ell x = P_\ell x$.  We conclude that $x = (I- P_\ell)x$. This implies $0 = L_\ell (I-P_\ell)x = (I- P_\ell)x =x $.\\

\noindent Next we show that $L_\ell$ is onto. Note $L_\ell (I-P_\ell)x = (I- P_\ell)x$, so we only need to show that for all $x$, $P_\ell x$ is in the range of $L_\ell$.  Define $$y = P_\ell A_\ell P_\ell x +P_\ell x-L P_\ell A_\ell P_\ell x.$$  We have $L_\ell ( P_\ell A_\ell P_\ell x) = L P_\ell A_\ell P_\ell x$ and
\begin{eqnarray}
L_\ell (P_\ell x-L P_\ell A_\ell P_\ell x) & = & L P_\ell x-L(P_\ell L P_\ell)(A_\ell P_\ell)x+P_\ell x - L P_\ell A_\ell P_\ell x\\
& = & P_\ell x - L P_\ell A_\ell P_\ell x,\nonumber
\end{eqnarray}
from which $L_\ell y = P_\ell x$.\\

\noindent $L_\ell$ is a continuous bijection between Hilbert spaces, and hence is an onto isomorphism by the open mapping theorem.
\end{proof}

\begin{lem}\label{secondone}
Define $L$ as in Lemma \ref{firstone}. For all $\ell>0$, $L_\ell$ is an onto isomorphism iff it is one to one.
\end{lem}
\begin{proof}[Proof of Lemma \ref{secondone}]
We only need to show that $P_\ell L P_\ell$ is one to one on $P_\ell H$ when $L_\ell$ is one to one on $H$, and apply Lemma \ref{firstone}.  Let $(P_\ell L P_\ell) P_\ell x =0.$  We have
\begin{eqnarray}
L_\ell [ P_\ell x - (I-P_\ell)L P_\ell x)] & = & L_\ell P_\ell x -L_\ell (I-P_\ell)L P_\ell x \nonumber\\
& = & L_\ell P_\ell x - (I-P_\ell)L P_\ell x\nonumber\\
& = & L_\ell P_\ell x-L P_\ell x = 0,\nonumber
\end{eqnarray}
where the last equality follows from $L_\ell P_\ell = L P_\ell$.  Since $L_\ell$ is one to one, we have that $P_\ell x = (I-P_\ell)LP_\ell x$, so that $P_\ell x = 0.$ 
\end{proof}

\begin{lem}\label{thirdone}
Let $(e_n)_{n \in \mathbb{N}}$ be an orthonormal basis for $H$, and $(f_n)_{n \in \mathbb{N}}$ be a Riesz basis for $H$, and $(k_\ell)_{\ell \in \mathbb{N}}\subset \mathbb{N}$ be an increasing sequence. Let $P_\ell$ be the orthogonal projection onto $\mathrm{span}\{ e_n\}_{n \leq k_\ell}$, then the following are equivalent:\\

\noindent 1) The operator $L$ is uniformly invertible with respect to $(P_\ell)_{\ell \in \mathbb{N}}$.\\
\noindent 2) For all $\ell >0$, $L_\ell$ is an onto isomorphism, and $\sup_{\ell \in \mathbb{N}} \|L_\ell^{-1} \| < \infty.$
\end{lem}

\begin{proof}[Proof of Lemma \ref{thirdone}]
\noindent 1) $\Longrightarrow$ 2): By Lemma \ref{firstone}, we only need to show that that $$\sup_{\ell \in \mathbb{N}} \|L_\ell^{-1} \| < \infty.$$  This follows from the identity\\
\begin{equation}\label{ident}
L_\ell^{-1} = [I-(I-P_\ell)L](P_\ell L P_\ell)^{-1}+I-P_\ell,
\end{equation}
which is hereby demonstrated:
\begin{eqnarray}
[I-(I-P_\ell)L](P_\ell L P_\ell)^{-1}+I-P_\ell & = & [I-(I-P_\ell)L]P_\ell L_\ell^{-1} P_\ell+I-P_\ell\\
& = & P_\ell L_\ell^{-1} P_\ell -L P_\ell L_\ell^{-1} P_\ell+I\nonumber\\
& = & (I-L)P_\ell L_\ell^{-1} P_\ell +I.\nonumber
\end{eqnarray}
We have $(I-L)P_\ell = I -L_\ell$, so
\begin{eqnarray}
[I-(I-P_\ell)L](P_\ell L P_\ell)^{-1}+I-P_\ell & = & (I-L) L_\ell^{-1} P_\ell +I\\
& = & L_\ell^{-1}P_\ell -P_\ell +I.\nonumber
\end{eqnarray}
Noting that $L_\ell (I-P_\ell) = I-P_\ell$, we obtain $L_\ell^{-1}P_\ell -P_\ell +I = L_\ell^{-1}$, which proves the identity.\\

\noindent 2) $\Longrightarrow$ 1): Noting that $(P_\ell L P_\ell)^{-1} = P_\ell L_\ell^{-1} P_\ell$ yields the result.\\
\end{proof}

\section{The first main result}\label{S:3}

\noindent We begin with some necessary definitions:\\

\begin{defin}\label{impdef}
Define $C_{\ell,d} = \{-\ell,\ldots,\ell\}^d$, and $e_n(x) = \frac{1}{(2\pi)^{d/2}}e^{i\langle x , n \rangle}$ for $n \in \mathbb{Z}^d$.  Let $P_\ell: L_2([-\pi,\pi]^d) \rightarrow L_2([-\pi,\pi]^d)$ be the orthogonal projection from $ L_2([-\pi,\pi]^d)$ onto $\mathrm{span}(e_n)_{n \in C_{\ell,d}}$.
\end{defin}

\noindent Let $(f_n)_{n \in \mathbb{Z}^d}$ be an exponential Riesz basis.  In the following sections, we abbreviate the statement ``$(f_n)_{n \in \mathbb{Z}^d}$ is a uniformly invertible Riesz basis with respect to the projections $(P_\ell)_{\ell \in \mathbb{N}}$ defined in definition \ref{impdef}'' by ``$(f_n)_{n \in \mathbb{Z}^d}$ is a uniformly invertible Riesz basis''.\\ 

\noindent To avoid confusion of indices, we write $t \in \mathbb{R}^d$ as $t=(t(1),\cdots,t(d))$.\\

\noindent For $\ell, d \in \mathbb{N}$ define the multivariate polynomial $$Q_{d,\ell}(t) = \prod_{k_1 =1}^\ell \Big(1-\frac{t(1)^2}{k_1^2}\Big)\cdot\ldots\cdot\prod_{k_d =1}^\ell \Big(1-\frac{t(d)^2}{k_d^2}\Big), \quad t = (t(1),\cdots , t(d)).$$

\noindent Here is the first main result of this paper.

\begin{thm}\label{maintheorem}
Let $(t_n)_{n \in \mathbb{Z}^d} \subset \mathbb{R}^d$, and define $f_n(x) = \frac{1}{(2\pi)^{d/2}}e^{i\langle x , t_n \rangle}$ for $n \in \mathbb{Z}^d$.  Let $(f_n)_{n \in \mathbb{Z}^d}$ be a uniformly invertible Riesz basis, then for all $f \in PW_{[-\pi,\pi]^d}$, there exists a unique sequence of polynomials $(\Psi_\ell)_{\ell \in \mathbb{N}}$,  $\Psi_\ell :\mathbb{R}^d \rightarrow \mathbb{R},$ such that\\

\noindent (a) $\Psi_\ell$ has coordinate degree at most $2 \ell$.\\
\noindent (b) $\Psi_\ell (t_n) = f(t_n)$ for all $n \in C_{\ell,d}$.\\
\noindent (c) $f(t) = \lim_{\ell \rightarrow \infty} \Psi_\ell(t) \frac{\mathrm{SINC} (\pi t)}{Q_{d,\ell} (t)}$, where the limit is both $L_2$ and uniform.
\end{thm}

\noindent Note: The expression in statement (c) of Theorem \ref{maintheorem} has removable singularities, but these can be evaluated by 
\begin{equation}
\lim_{t \rightarrow n} \frac{\mathrm{sinc} (\pi t)}{Q_{1,\ell} (t)} = \frac{(\ell!)^2}{(\ell+n)!(\ell-n)!},\quad n \in \{-\ell,\ldots \ell \}.\nonumber
\end{equation}

\noindent The proof of Theorem \ref{maintheorem} requires several lemmas, beginning with the following equivalence between the existence of particular Riesz bases and a polynomial interpolation condition:

\begin{lem}\label{othermain}
Let $(t_n)_{n \in \mathbb{Z}^d} \subset \mathbb{R}^d$ where  $f_n(x) = \frac{1}{(2\pi)^{d/2}}e^{i\langle x , t_n \rangle}$ and $e_n(x) = \frac{1}{(2\pi)^{d/2}}e^{i\langle x , n \rangle}$.  The sequence $ (f_{\ell,n})_{n \in \mathbb{Z}^d} := (f_n)_{n \in C_{\ell,d}} \cup (e_n)_{n \notin C_{\ell,d}}$ is a Riesz basis for $L_2([-\pi,\pi]^d)$ iff the following conditions hold.\\

\noindent 1) For all $n \in C_{\ell,d}$, $t_n \notin (\mathbb{Z}\setminus \{-\ell,\cdots,\ell\})^d$.\\
\noindent 2) For any sequence $(c_k)_{k \in C_{\ell,d}}$, there exists a unique polynomial $\Psi_\ell$ with coordinate degree at most $2\ell$ such that $\Psi_\ell(t_k)= c_k$ for $k \in C_{\ell,d}$.\\
\end{lem}

\begin{proof}[Proof of Lemma \ref{othermain}]
Suppose that the sequence $(f_n)_{n \in \mathbb{Z}^d}$ is a Riesz basis for $L_2([-\pi,\pi]^d)$.  We compute the biorthogonal functions of $(f_{\ell,n})_{n \in \mathbb{Z}^d}$ when $n \in C_{\ell,d}$:
\begin{eqnarray}
f_{\ell,n}^*  & = & \sum_{k \in \mathbb{Z}^d} \langle f_{\ell,n}^* , e_k \rangle e_k = \sum_{k \in C_{\ell,d}} \langle f_{\ell,n}^* , e_k \rangle e_k + \sum_{k \notin C_{\ell,d}}  \langle f_{\ell,n}^* , e_k \rangle e_k\nonumber\\
& = & \sum_{k \in C_{\ell,d}} \langle f_{\ell,n}^* , e_k \rangle e_k.\nonumber
\end{eqnarray}
Passing to the Fourier transform and defining $G_{\ell, n} = \mathcal{F}(f_{\ell,n}^*)$, we have
\begin{eqnarray}\label{firstsum}
G_{\ell,n}(t) & = & \sum_{k \in C_{\ell,d}} G_{\ell,n}(k)\mathrm{SINC}\pi(t-k)\\
& = & \bigg( \sum_{k \in C_{\ell,d}}\frac{G_{\ell,n}(k)(-1)^{k(1)+\ldots+k(d)} t(1)\cdot\ldots\cdot t(d)}{(t(1)-k(1))\cdot \ldots\cdot(t(d)-k(d))} \bigg)\mathrm{SINC}(\pi t), \quad t \in \mathbb{R}^d.\nonumber
\end{eqnarray}
Denote the $k^{th}$ summand in equation (\ref{firstsum}) by $A_k$, then
\begin{eqnarray}
A_{\ell,n,k}\!\! & = &\!\! A_{\ell,n,k} \frac{\prod_{j_1 = -\ell \atop j_1 \neq k(1)}^\ell (t(1)-j_1) \cdot \ldots \cdot \prod_{j_d = -\ell \atop j_d \neq k(d)}^\ell (t(d)-j_d)}{\prod_{j_1 = -\ell \atop j_1 \neq k(1)}^\ell (t(1)-j_1) \cdot \ldots \cdot \prod_{j_d = -\ell \atop j_d \neq k(d)}^\ell (t(d)-j_d)}\nonumber\\
& = &\!\! \frac{G_{\ell,n}(k)(-1)^{k(1)+\ldots+k(d)}t(1)\cdot\ldots\cdot t(d) \prod_{j_1 = -\ell \atop j_1 \neq k(1)}^\ell (t(1)-j_1) \cdot \ldots \cdot \prod_{j_d = -\ell \atop j_d \neq k(d)}^\ell (t(d)-j_d)}{\prod_{j_1 = -\ell}^\ell (t(1)-j_1) \cdot \ldots \cdot \prod_{j_d = -\ell}^\ell (t(d)-j_d)}\nonumber\\
& = &\!\! \frac{G_{\ell,n}(k)\frac{1}{(\ell!)^2}(-1)^{k(1)+\ldots+k(d) +\ell d}\prod_{j_1 = -\ell \atop j_1 \neq k(1)}^\ell (t(1)-j_1) \cdot \ldots \cdot \prod_{j_d = -\ell \atop j_d \neq k(d)}^\ell (t(d)-j_d)}{\prod_{j_1 = 1}^\ell \Big(1-\frac{t(1)^2}{j_1^2}\Big) \cdot \ldots \cdot \prod_{j_d = 1}^\ell \Big(1-\frac{t(d)^2}{k_d^2}\Big)}\nonumber\\
& = &\!\! \frac{p_{\ell,n,k} (t)}{Q_{d,\ell} (t)},\nonumber
\end{eqnarray}
where $p_{\ell,n,k}$ is some polynomial with coordinate degree at most $2\ell$. Substituting into equation ({\ref{firstsum}}), we obtain
$$G_{\ell,n}(t) = \Big(\sum_{k \in C_{\ell,d}} p_{\ell,n,k}(t) \Big)\frac{\mathrm{SINC}(\pi t)}{Q_{d,\ell} (t)}:= \phi_{\ell,n} (t)\frac{\mathrm{SINC}(\pi t)}{Q_{d,\ell} (t)},$$ where $\phi_{\ell,n}$ is a polynomial having coordinate degree at most $2 \ell$.  The fact that each zero of $\mathrm{sinc}(\pi z)$ has multiplicity one implies that the zero set of $\frac{\mathrm{SINC}(\pi t)}{Q_{d,\ell} (t)}$ (which is entire) is $(\mathbb{Z}\setminus \{-\ell,\cdots, \ell\})^d \subset \mathbb{C}$. Using that $\mathrm{SINC}\pi(x-y)$ is the reproducing kernel for $PW_{[-\pi,\pi]^d}$, we see that  $$G_{\ell,n}(t_m) = \delta_{nm},$$ for all $n,m \in C_{\ell,d}$. This yields that $$1 = \phi_{\ell,n} (t_n)\Big(\frac{\mathrm{SINC}(\pi t)}{Q_{d,\ell} (t)}\Big)\Big|_{t_n}.$$ This shows $\phi_{\ell,n} (t_n) \neq 0$ and $\frac{\mathrm{SINC}(\pi t_n)}{Q_{d,\ell} (t_n)}\neq0$ for $n \in C_{\ell,d}$, and that $t_n \notin (\mathbb{Z}\setminus \{-\ell,\cdots, \ell\})^d$ (statement 1) of Lemma \ref{othermain}).\\

\noindent For $n,m \in C_{\ell,d}, n\neq m, $ $$0 = G_{\ell,n}(t_m) = \phi_{\ell,n} (t_m)\Big(\frac{\mathrm{SINC}(\pi t)}{Q_{d,\ell} (t)}\Big)\Big|_{t_m},$$ We conclude that 
\begin{displaymath}
   \phi_{\ell,n}(t_m) = \left\{
     \begin{array}{lr}
       \frac{Q_{d,\ell}(t_n)}{\mathrm{SINC}\pi t_n} \neq 0,  &  n = m\\
       0, & n \neq m
     \end{array}
   \right.
\end{displaymath}
for $n,m \in C_{\ell,d}$ From this, the ``existence'' part of statement 2) in Lemma \ref{othermain} readily follows.  Restated, the evaluation map taking the space of all polynomials of coordinate degree at most $2 \ell$ to $\mathbb{R}^{(2\ell+1)^d}$ is onto.  These spaces have the same dimension, hence the evaluation map is a bijection, which completes the proof of statement 2).\\

\noindent Suppose that 1) and 2) hold.  For $n \in C_{\ell,d}$, let $p_{\ell,n}$ be the unique polynomial of coordinate degree at most $2\ell$ such that $p_{\ell,n} (t_m) = \delta_{nm}$ for $m \in C_{\ell,d}$.  Define 
\begin{equation}\label{biorth}
\Phi_{\ell,n} (t) = \frac{Q_{d,\ell} (t_n) \mathrm{SINC}\pi t}{Q_{d,\ell}(t) \mathrm{SINC}\pi t_n}p_{\ell,n}(t).
\end{equation}
Partial fraction decomposition can be used to show that $\Phi_{\ell,n} \in PW_{[-\pi,\pi]^d}$.  For $n,m \in C_{\ell,d}$, we therefore have $$\delta_{n,m} = \langle \Phi_{\ell,n} (\cdot) , \mathrm{SINC}\pi((\cdot)-t_m) \rangle = \langle \mathcal{F}^{-1}(\Phi_{\ell,n}) , f_m \rangle.$$
Let $L_\ell$ be defined as before.  Let $f=\sum_{n \in \mathbb{Z}^d} c_n e_n$ such that $L_\ell (f) = 0$, then $$0 = \sum_{n \in C_{\ell,d}} c_n f_n+\sum_{n \notin C_{\ell,d}}c_n e_n.$$  If, for each $n \in C_{\ell,d}$ we integrate the above equation against $\mathcal{F}^{-1}(\Phi_{\ell,n})$, we see that $c_k = 0$ for $k \in C_{\ell,d}$, so that $c_k =0$ for all $k \in \mathbb{Z}^d$.  $L_\ell$ is one to one, so by Lemma \ref{secondone}, it is an onto isomorphism from $L_2([-\pi,\pi]^d)$ to itself.
\end{proof}

\begin{proof}[Proof of (a) and (b) of Theorem \ref{maintheorem}]
Lemmas \ref{thirdone} and \ref{othermain} imply the existence of a unique sequence of polynomials satisfying statements (a) and (b) of Theorem \ref{maintheorem}, namely,
$$\Psi_\ell(t) = \sum_{n \in C_{\ell,d}} f(t_n) p_{\ell,n}(t),$$ where $p_{\ell,n}$ is defined as in the proof of Lemma \ref{othermain}.\\
\end{proof}

\noindent It remains to show that this sequence of polynomials satisfies the statement (c) of Theorem \ref{maintheorem}.

\begin{prop}\label{fourthone}
\noindent The following are equivalent.\\

\noindent 1) $L$ is uniformly invertible with respect to the projections $(P_\ell)_{\ell \in \mathbb{N}}$.\\
\noindent 2) For all $x \in H$, $\lim_{\ell \rightarrow \infty} (L_\ell^*)^{-1}(I-P_\ell)x=0$.
\end{prop}

\begin{proof}[Proof of Proposition \ref{fourthone}]
\noindent It is clear that 1) implies 2). For the other direction, note that the equality $L_\ell^* = P_\ell L^* +I - P_\ell$ implies that
\begin{equation}\label{prebiorthapprox}
I = (L_\ell^*)^{-1} P_\ell L^* +(L_\ell^*)^{-1}(I-P_\ell).
\end{equation}
This implies that $(L_\ell^*)^{-1}P_\ell$ is pointwise bounded.  Together with the assumption in 2), this implies $(L_\ell^*)^{-1}$ is pointwise bounded, hence norm bounded by the uniform boundedness principle.  This yields uniform invertibility of $L$.
\end{proof}

\begin{lem}\label{blip}
\noindent The following are equivalent:\\

\noindent 1) For all $g \in L_2([-\pi,\pi]^d)$, we have
\begin{equation}\label{biorthapprox}
g = \lim_{\ell \rightarrow \infty} (L_\ell^*)^{-1} P_\ell L^* g.
\end{equation}
\noindent 2) $L$ is uniformly invertible.
\end{lem}
\begin{proof}[Proof of Lemma \ref{blip}]
Recall equation (\ref{prebiorthapprox}) and apply Proposition \ref{fourthone}.
\end{proof}

\begin{prop}\label{blip2}
Statement (c) of Theorem \ref{maintheorem} is true iff
\begin{eqnarray}\label{equivlimit}
0 = \lim_{\ell \rightarrow \infty} \sum_{n \in C_{\ell,d}} |f(t_n)|^2\bigg[1- \frac{\mathrm{SINC} \pi t_n}{Q_{d,\ell} (t_n)} \bigg]^2:=  \lim_{\ell \rightarrow \infty} S_{\ell,d}, \quad f\in PW_{[-\pi,\pi]^d}.
\end{eqnarray}
\end{prop}

\begin{proof}[Proof of Proposition \ref{blip2}]
\noindent Note that $Le_n = f_n$ implies that $f_n^* = (L^*)^{-1}e_n.$  Similarly, $f_{\ell,n}^* = (L_\ell^*)^{-1}e_n$.  Given $f \in PW_{[-\pi,\pi]^d}$, let $g= \mathcal{F}^{-1}(f)$.  Equation (\ref{biorthapprox}) shows:
\begin{eqnarray}
\mathcal{F}^{-1}(f) & = & \lim_{\ell \rightarrow \infty} (L_\ell^*)^{-1} \sum_{n \in C_{\ell,d}} \langle L^* g, e_n \rangle e_n = \lim_{\ell \rightarrow \infty} (L_\ell^*)^{-1} \sum_{n \in C_{\ell,d}} \langle g, f_n \rangle e_n\nonumber\\
& = & \lim_{\ell \rightarrow \infty} \sum_{n \in C_{\ell,d}} \langle g, f_n \rangle f_{\ell,n}^* = \lim_{\ell \rightarrow \infty} \sum_{n \in C_{\ell,d}} f(t_n) f_{\ell,n}^*.\nonumber
\end{eqnarray}
Passing to the Fourier transform, we have
\begin{equation}\label{blah}
f = \lim_{\ell \rightarrow \infty} \sum_{n \in C_{\ell,d}} f(t_n) \mathcal{F}(f_{\ell,n}^*), \quad f \in PW_{[-\pi,\pi]^d},
\end{equation}
where convergence in both $L_2$ and uniform.  The values of a function in $PW_{[-\pi,\pi]^d}$ on the set $(t_n)_{n \in \mathbb{Z}^d}$ uniquely determine the function.  This and equation (\ref{biorth}) show that $$\mathcal{F}(f_{\ell,n}^*)(t) = G_{\ell,n}(t) = \frac{Q_{d,\ell} (t_n) \mathrm{SINC}\pi t}{Q_{d,\ell}(t) \mathrm{SINC}\pi t_n}p_{\ell,n}(t), \quad n \in C_{\ell,d}.$$ %Equation (\ref{blah}) becomes
%\begin{eqnarray}
%f(t) & = & \lim_{\ell \rightarrow \infty} \bigg(\sum_{n \in C_{\ell,d}} f(t_n)\frac{Q_\ell^d (t_n)}{\mathrm{SINC}\pi t_n}p_{\ell,n}(t)\bigg). \frac{\mathrm{SINC}\pi t}{Q_\ell^d(t)}, \quad f \in PW_{[-\pi,\pi]^d}\nonumber\\
%& = &  \lim_{\ell \rightarrow \infty} I_\ell (f)
%\end{eqnarray}
This implies that
\begin{eqnarray}
\Psi_\ell(t) \frac{\mathrm{SINC} (\pi t)}{Q_{d,\ell} (t)} & = & \Big(\sum_{n \in C_{\ell,d}} f(t_n) p_{\ell,n}(t)\Big)  \frac{\mathrm{SINC} (\pi t)}{Q_{d,\ell} (t)}\nonumber\\
& = & \sum_{n \in C_{\ell,d}} f(t_n) \frac{\mathrm{SINC} \pi t_n}{Q_{d,\ell} (t_n)}\mathcal{F}(f_{\ell,n}^*)(t).\nonumber
\end{eqnarray}
Combined with equation (\ref{blah}), we see that statement (c) of Theorem \ref{maintheorem} holds iff
\begin{eqnarray}
0  =  \lim_{\ell \rightarrow \infty} \sum_{n \in C_{\ell,d}} f(t_n)\bigg[1- \frac{\mathrm{SINC} \pi t_n}{Q_{d,\ell} (t_n)} \bigg] \mathcal{F}(f_{\ell,n}^*),  \quad f\in PW_{[-\pi,\pi]^d},\nonumber
\end{eqnarray}
where the limit is in the $L_2$ sense. Passing to the inverse Fourier transform, statement (c) holds iff
\begin{eqnarray}\label{blip3}
0 = \lim_{\ell \rightarrow \infty} (L_\ell^*)^{-1}\bigg( \sum_{n \in C_{\ell,d}} f(t_n)\bigg[1- \frac{\mathrm{SINC} \pi t_n}{Q_{d,\ell} (t_n)} \bigg] e_n \bigg), \quad f\in PW_{[-\pi,\pi]^d}.
\end{eqnarray}
In addition to having uniformly norm bounded inverses, $L_\ell$ is pointwise bounded, so there exists $m, M>0$ such that for all $\ell \in \mathbb{N}$, $f \in PW_{[-\pi,\pi]^d}$, $m\| f\| \leq \|(L_\ell^*)^{-1} f \| \leq M\| f \|$.  This, combined with equation (\ref{blip3}), proves the proposition.
\end{proof}

\begin{prop}\label{sinclimitstuff}
\noindent The following are true:\\

\noindent 1) $\sup_{x \in \mathbb{R}} \sup_{\ell \in \mathbb{N}} \Big|\frac{\mathrm{sinc}(\pi x)}{Q_{1,\ell}(x)} \Big|=1.$\\

\noindent 2) Define $\Delta_{\ell,d} = \Big\{ n \in \mathbb{Z}^d \big| \Big\| \frac{t_n}{\ell+1} \Big\|_{\infty} < \frac{1}{\ell^{2/3}} \Big\}$ for $\ell \in \mathbb{N}$, then
\begin{eqnarray}
0 \leq 1-\frac{\mathrm{SINC}(\pi t_n)}{Q_{d,\ell} (t_n)} < 1-e^{\frac{-d(\ell+2)}{\ell^{4/3}-1}}, \quad n \in \Delta_{\ell,d}.
\end{eqnarray}
\end{prop}

\begin{proof}[Proof of Proposition \ref{sinclimitstuff}]
\noindent For 1), the identity $$\mathrm{sinc}(\pi t) = \prod_{k=1}^\infty \Big( 1-\frac{t^2}{k^2} \Big)$$ implies
\begin{equation}
\frac{\mathrm{sinc}(\pi t)}{Q_{1,\ell} (t)} = \prod_{k=\ell+1}^\infty \Big(1-\frac{t^2}{k^2} \Big),
\end{equation}
where convergence is uniform on compact subsets of $\mathbb{C}$.  Fix $\ell \in \mathbb{N}$.  If $t \in [0,\ell+1]$, then $\Big|\frac{\mathrm{sinc}(\pi t)}{Q_{1,\ell}(t)} \Big| \leq1.$  Note that $|Q_{1,\ell} (t)| = \prod_{k=1}^\ell \Big(\frac{t^2}{k^2} -1 \Big)$ is increasing on $(\ell+1, \infty)$.  If $t \in (\ell+1, \infty)$, then
\begin{equation}
\Big|\frac{\mathrm{sinc}(\pi t)}{Q_{1,\ell}(t)} \Big| = \Big|\frac{\sin(\pi t)}{\pi t Q_{1,\ell}(t)} \Big| < \frac{1}{\pi(\ell+1)| Q_{1,\ell} (\ell+1)|}.\nonumber
\end{equation}
Computation yields
\begin{equation}
| Q_{1,\ell} (\ell+1)| = \frac{(2\ell+1)!}{\ell! (\ell+1)!},\nonumber
\end{equation}
so $$ \Big|\frac{\mathrm{sinc}(\pi t)}{Q_{1,\ell}(t)} \Big| < \frac{(\ell!)^2}{\pi(2\ell+1)!} <1 .$$
Observing that $\frac{\mathrm{sinc}(\pi t)}{Q_{1,\ell}(t)}$ is even proves 1).\\

\noindent For 2), let $t \in \mathbb{R}$ such that $\Big| \frac{t}{\ell+1} \Big| < \frac{1}{\ell^{2/3}}$, then $0 < \frac{\mathrm{sinc}(\pi t)}{Q_{1,\ell}(t)}$, and
\begin{eqnarray}\label{logstuff}
-\log\Big( \frac{\mathrm{sinc}(\pi t)}{Q_{1,\ell}(t)} \Big) &  = & -\sum_{k= \ell+1}^\infty \log\Big(1- \frac{t^2}{k^2} \Big) = \sum_{k= \ell+1}^\infty \sum_{j=1}^\infty \frac{t^2}{j k^{2j}}\\
& = & \sum_{j=1}^\infty \frac{1}{j} \Big(\sum_{k= \ell+1}^\infty \frac{1}{k^{2j}} \Big) t^{2j}.\nonumber
\end{eqnarray}
Basic calculus shows that $$\sum_{k= \ell+1}^\infty \frac{1}{k^{2j}} <\frac{1}{(\ell+1)^{2j}}+\frac{1}{(2j-1)(\ell+1)^{2j-1}}.$$  Equality (\ref{logstuff}) implies
\begin{eqnarray}\label{morelogstuff}
-\log\Big( \frac{\mathrm{sinc}(\pi t)}{Q_{1,\ell}(t)} \Big) & < &\sum_{j=1}^\infty \frac{1}{j}\Big(\frac{t}{\ell+1}\Big)^{2j}+  (\ell+1)\sum_{j=1}^\infty \frac{1}{j(2j-1)}\Big(\frac{t}{\ell+1}\Big)^{2j}\\
& < & (\ell+2)\sum_{j=1}^\infty \Big(\frac{t}{\ell+1}\Big)^{2j} < \frac{\ell+2}{\ell^{4/3}-1}.\nonumber
\end{eqnarray}
If $n \in \Delta_{\ell,d}$, then for each $1 \leq k \leq d$, $\Big| \frac{t_n (k)}{\ell+1} \Big| >\frac{\ell+2}{\ell^{4/3}-1}$, so that
\begin{equation}
\log \Big(\frac{\mathrm{SINC}(\pi t_n)}{Q_{d,\ell}(t_n)} \Big)  = \sum_{k=1}^d \log \Big( \frac{\mathrm{sinc}(\pi t_n(k))}{Q_{1,\ell}(t_n(k))} \Big) > -\frac{d(\ell+2)}{\ell^{4/3}-1}.\nonumber
\end{equation}
Statement 2) of Proposition \ref{sinclimitstuff} follows readily.
\end{proof}

\begin{proof}[Proof of statement (c) in Theorem \ref{maintheorem}]
Proposition \ref{sinclimitstuff} gives the following:\\
\begin{eqnarray}
S_{\ell,d} & \leq & \Big(\sum_{n \in \Delta_{\ell,d}} + \sum_{n \in \mathbb{Z}^d \setminus\Delta_{\ell,d}}\Big) |f(t_n)|^2\bigg[1- \frac{\mathrm{SINC} \pi t_n}{Q_{d,\ell} (t_n)} \bigg]^2 \nonumber\\
& \leq &  \bigg( 1-e^{\frac{-d(\ell+2)}{\ell^{4/3}-1}} \bigg)^2 \sum_{n \in \mathbb{Z}^d} |f(t_n)|^2 + \sum_{n:\ \frac{\ell+1}{\ell^{2/3}}\leq \| t_n \|_\infty} 4 | f(t_n) |^2.
\end{eqnarray}
Now $(f(t_n))_{n \in \mathbb{Z}^d} \in \ell_2(\mathbb{Z}^d)$ implies that $\lim_{\ell \rightarrow \infty} S_{\ell,d} =0$, so that by Proposition \ref{blip2}, statement (c) in Theorem \ref{maintheorem} is true.
\end{proof}

\section{The second main result}\label{S:4}

\noindent Theorem \ref{maintheorem} can be simplified.  The function $$t \mapsto \frac{\mathrm{SINC} (\pi t)}{Q_{d,\ell} (t)}$$ becomes more computationally complex for large values of $\ell$.  If, at the cost of global $L_2$ and uniform convergence, we adopt an approximation 
\begin{equation}
\mathrm{SINC} (\pi t) \simeq Q_{d,\ell} (t) e^{-\sum_{k=1}^{N} \frac{1}{k(2k-1)}\frac{\| t \|_{2k}^{2k}}{(\ell+1/2)^{2k-1}}},
\end{equation}
we bypass this difficulty as the exponent of the above quantity is simply a rational function of $\ell >0$.  This is precisely quantified in Theorem \ref{simplethm}, which is the second main result of this paper.

\begin{thm}\label{simplethm}
Let $(t_n)_{\mathbb{Z}^d} \subset \mathbb{R}^d$ be a sequence such that the associated exponential functions are a uniformly invertible Riesz basis for $L_2([-\pi,\pi]^d)$.  For $N \in \{0,1,2,\ldots \}$, $A>0$, let $E_{\ell,N,A} = [-A (\ell+1/2)^\frac{2N+1}{2N+2},A (\ell+1/2)^\frac{2N+1}{2N+2}]$.  Let $f \in PW_{[-\pi,\pi]^d}$ where $(\Psi_\ell)_\ell$ is the sequence of interpolating polynomials from Theorem \ref{maintheorem}, and let $N \geq 0$.  Let $f \in PW_{[-\pi,\pi]^d}$, and define $$I_{f,\ell} (t) = \Psi_\ell(t) e^{-\sum_{k=1}^{N} \frac{1}{k(2k-1)}\frac{\| t \|_{2k}^{2k}}{(\ell+1/2)^{2k-1}}},$$ then
\begin{equation}\label{simplecor2}
\lim_{\ell \rightarrow \infty} \Big\| f(t) - I_{f,\ell}(t) \Big\|_{L_2((E_{\ell,N,A})^d)}=0,
\end{equation}
and
\begin{equation}\label{simplecor1}
\lim_{\ell \rightarrow \infty} \Big\| f(t) - I_{f,\ell} (t) \Big\|_{L_\infty((E_{\ell,N,A})^d)}=0.
\end{equation}
\end{thm}

\noindent If $N=0$ in Theorem \ref{simplethm}, we have the following extension of Corollary \ref{lyuseippavcor} to arbitrary multivariate bandlimited functions (at the expense of introducing uniform invertibility):
\begin{cor}\label{simplecormod}
For all $f \in PW_{[-\pi,\pi]^d}$, we have
\begin{equation}%\label{simplecor2mod}
\lim_{\ell \rightarrow \infty} \Big\| f(t) - \Psi_\ell(t)\Big\|_{L_2([-A(\ell+1/2)^{1/2},A(\ell+1/2)^{1/2}]^d)}=0,
\end{equation}
and
\begin{equation}\label{simplecor1mod}
\lim_{\ell \rightarrow \infty} \Big\| f(t) - \Psi_\ell(t)\Big\|_{L_\infty([-A(\ell+1/2)^{1/2},A(\ell+1/2)^{1/2}]^d)}=0.
\end{equation}
\end{cor}
\noindent It is evident that if $(t_n)_{n \in \mathbb{Z}^d}\subset \mathbb{R}^d$ is any subset such that the associated sequence of exponentials $(f_n)_{n \in \mathbb{Z}^d}$ is a Riesz basis for $L_2([-\pi,\pi]^d)$, then the map $f \mapsto (f(t_n))_{n \in \mathbb{Z}^d}$ is a bijection from $PW_{[-\pi,\pi]^d}$ to $\ell_2(\mathbb{Z}^d)$.  This allows for a nice interpretation of Corollary \ref{simplecormod}.  Given a sequence $(t_n)_{n \in \mathbb{Z}^d}\subset \mathbb{R}^d$ (subject to the hypotheses of Theorem \ref{simplethm}), and sampled data $\big( (t_n,c_n) \big)_{n \in \mathbb{Z}^d}$ where $(c_n)_{n \in \mathbb{Z}^d} \in \ell_2(\mathbb{Z}^d)$, then a unique sequence of Lagrangian polynomial interpolants exists, and in the appropriate limit, converges to the unique band-limited interpolant of the same data.\\

\noindent When $N=1$, we have a sampling theorem with a Gaussian multiplier: $$ f(t) \simeq \Psi_\ell (t) e^{-\frac{\| t\|_2^2}{(\ell+1/2)}}, \quad f \in PW_{[-\pi,\pi]^d}.$$  Compare Theorem \ref{simplethm} with Theorem 2.6 in \cite{BSS}, which is a multivariate sampling theorem with a Gaussian multipler with global $L_2$ and uniform convergence.  Also compare Theorem \ref{simplethm} with Theorem 2.1 in \cite{SS}, which, when $d=1$ and the data nodes are equally spaced, gives another recovery formula involving a Gaussian mulitplier in the context of over-sampling.\\

\noindent The proof of Theorem \ref{simplethm} relies on two lemmas:

\begin{lem}\label{techlem1}
Let $d \in \mathbb{N}$, $N \in \{0,1,2,\ldots \}$, and $A > 0$.  There exists $M >0$ such that for sufficiently large $\ell$, and any $t \in (E_{\ell,N,A})^d$, we have
\begin{eqnarray}
& &\Big|Q_{d,\ell}(t)  e^{-\sum_{k=1}^{N} \frac{1}{k(2k-1)}\frac{\| t \|_{2k}^{2k}}{(\ell+1/2)^{2k-1}}} - e^\frac{\| t\|_{2(N+1)}^{2(N+1)}}{(\ell+1/2)^{2N+1}(N+1)(2N+1)} \mathrm{SINC}(\pi t)\Big|\nonumber\\
& &\quad\leq M(\ell+1/2)^{-\frac{1}{N+1}}|\mathrm{SINC}(\pi t)|.\nonumber
\end{eqnarray}
\end{lem}

\begin{lem}\label{techlem2}
For all $f \in PW_{[-\pi,\pi]^d}$ and $N \in \{0,1,2,\ldots \}$, we have
\begin{equation}
\lim_{\ell \rightarrow \infty} \sup_{t \in (E_{\ell,N,A})^d} \Big| \Big(e^\frac{\| t\|_{2(N+1)}^{2(N+1)}}{(\ell+1/2)^{2N+1}(N+1)(2N+1)}  -1 \Big)f(t) \Big| = 0.\nonumber
\end{equation}
\end{lem}

\noindent The proof of Lemma \ref{techlem1} relies on the following proposition.

\begin{prop}\label{blip5}
If $f:(0,\infty) \rightarrow (0,\infty)$ is convex, decreasing, differentiable, and integrable away from $0$, then
$$\frac{1}{4}f'(\ell+1/2)\leq \sum_{k=\ell+1}^\infty f(k) -\int_{\ell+1/2}^\infty f(x) dx \leq 0.$$
\end{prop}
\begin{proof}[Proof of Proposition \ref{blip5}]
The proof follows naturally from geometric considerations.
\end{proof}

\begin{proof}[Proof of Lemma \ref{techlem1}]
Letting $|t| < \ell+1/2$ and recalling equation (\ref{logstuff}), we see that
\begin{eqnarray}\label{blork2}
& &-\log\Big(\frac{\mathrm{sinc}(\pi t)}{Q_{1,\ell} (t)}\Big)-\sum_{k=1}^{\infty} \frac{1}{k(2k-1)}\frac{t^{2k}}{(\ell+1/2)^{2k-1}}\\
& & \quad = \sum_{k=1}^\infty \Big[ \sum_{j=\ell+1}^\infty \frac{1}{j^{2k}} -\frac{1}{(2k-1)(\ell+1/2)^{2k-1}} \Big] \frac{t^{2k}}{k}.\nonumber
\end{eqnarray}
Applying Proposition \ref{blip5} to the function $f(t) = \frac{1}{t^{2k}}$ when $k \geq 1$, we obtain
\begin{equation}
\frac{-k}{2(\ell+1/2)^{2k+1}} \leq  \sum_{j=\ell+1}^\infty \frac{1}{j^{2k}} -\frac{1}{(2k-1)(\ell+1/2)^{2k-1}} \leq 0.\nonumber
\end{equation}
Equation (\ref{blork2}) becomes
\begin{eqnarray}
\frac{-1}{2(\ell+1/2)}\sum_{k=1}^{\infty} \Big(\frac{t}{\ell+1/2}\Big)^{2k} \leq -\log\Big(\frac{\mathrm{sinc}(\pi t)}{Q_{1,\ell} (t)}\Big)-\sum_{k=1}^{\infty} \frac{1}{k(2k-1)}\frac{t^{2k}}{(\ell+1/2)^{2k-1}} \leq 0.\nonumber
\end{eqnarray}
Restated,
\begin{eqnarray}
& &-\frac{1}{2(\ell+1/2)}\frac{\Big( \frac{t}{\ell+1/2} \Big)^2}{1-\Big(\frac{t}{\ell+1/2} \Big)^2}+\sum_{k=N+1}^\infty \frac{1}{k(2k-1)}\frac{t^{2k}}{(\ell+1/2)^{2k-1}}\\
& &\quad\leq -\log\Big(\frac{\mathrm{sinc}(\pi t)}{Q_{1,\ell} (t)}\Big)-\sum_{k=1}^{N} \frac{1}{k(2k-1)}\frac{t^{2k}}{(\ell+1/2)^{2k-1}}\nonumber\\
& &\quad\leq  \sum_{k=N+1}^\infty \frac{1}{k(2k-1)}\frac{t^{2k}}{(\ell+1/2)^{2k-1}}.\nonumber
\end{eqnarray}
Exponentiating,
\begin{eqnarray}\label{blork3}
& & e^{\bigg(-\frac{1}{2(\ell+1/2)}\frac{\big( \frac{t}{\ell+1/2} \big)^2}{1-\big(\frac{t}{\ell+1/2} \big)^2}\bigg)} e^{\sum_{k=N+1}^\infty \frac{1}{k(2k-1)}\frac{t^{2k}}{(\ell+1/2)^{2k-1}}}\\
& &\quad \leq\frac{Q_{1,\ell} (t)e^{-\sum_{k=1}^N \frac{1}{k(2k-1)}\frac{t^{2k}}{(\ell+1/2)^{2k-1}}}}{\mathrm{sinc}(\pi t)} \leq  e^{\sum_{k=N+1}^\infty \frac{1}{k(2k-1)}\frac{t^{2k}}{(\ell+1/2)^{2k-1}}}.\nonumber
\end{eqnarray}
Let $\ell$ be chosen large enough so that $A (\ell+1/2)^\frac{2N+1}{2N+2} <\ell+1/2$. If $\ell$ is large enough, then for any $t \in E_{\ell N,A}$, $t=c (\ell+1/2)^\frac{2N+1}{2N+2}$ for some $c \in [-A,A]$.  For such $t$, inequality (\ref{blork3}) implies
\begin{eqnarray}
& & e^{\bigg(-\frac{1}{2(\ell+1/2)^\frac{N+2}{N+1}}\frac{c^2}{1-c^2(\ell+1/2))^\frac{-1}{N+1}}\bigg)} e^{\sum_{k=N+1}^\infty \frac{c^{2k}}{k(2k-1)}(\ell+1/2)^{\big(1-\frac{k}{N+1}\big)}}\nonumber\\
& & \quad\leq \frac{Q_{1,\ell} (t)e^{-\sum_{k=1}^N \frac{1}{k(2k-1)}\frac{t^{2k}}{(\ell+1/2)^{2k-1}}}}{\mathrm{sinc}(\pi t)} \leq  e^{\sum_{k=N+1}^\infty \frac{c^{2k}}{k(2k-1)}(\ell+1/2)^{\big(1-\frac{k}{N+1}\big)}}.\nonumber
\end{eqnarray}
If $t \in (E_{\ell,N,A})^d$, then $t=c (\ell+1/2)^\frac{2N+1}{2N+2}$ for some $c \in [-A,A]^d$.  For any such $t$, we have
\begin{eqnarray}\label{blork4}
& & e^{\bigg(-\frac{d}{2(\ell+1/2)^\frac{N+2}{N+1}}\frac{A^2}{1-A^2(\ell+1/2))^\frac{-1}{N+1}}\bigg)} e^{\sum_{k=N+1}^\infty \frac{\|c\|_{2k}^{2k}}{k(2k-1)}(\ell+1/2)^{\big(1-\frac{k}{N+1}\big)}}\\
& & \quad\leq\frac{Q_{d,\ell} (t)e^{-\sum_{k=1}^N \frac{1}{k(2k-1)}\frac{\|t\|_{2k}^{2k}}{(\ell+1/2)^{2k-1}}}}{\mathrm{SINC}(\pi t)} \leq  e^{\sum_{k=N+1}^\infty \frac{\|c\|_{2k}^{2k}}{k(2k-1)}(\ell+1/2)^{\big(1-\frac{k}{N+1}\big)}}.\nonumber
\end{eqnarray}
On one hand,
\begin{eqnarray}\label{upperinequall}
e^{\sum_{k=N+1}^\infty \frac{\|c\|_{2k}^{2k}}{k(2k-1)}(\ell+1/2)^{\big(1-\frac{k}{N+1}\big)}} \leq e^{\Big( \frac{\|c\|_{2(N+1)}^{2(N+1)}}{(N+1)(2N+1)} + O\big( (\ell+1/2)^\frac{-1}{N+1} \big) \Big)}
\end{eqnarray}
where the ``big O'' constant is independent of $c \in [-A,A]^d$.
\noindent On the other hand,
\begin{eqnarray}\label{lowerinequal}
e^{\frac{\|c\|_{2(N+1)}^{2(N+1)}}{(N+1)(2N+1)}} \leq e^{\sum_{k=N+1}^\infty \frac{\|c\|_{2k}^{2k}}{k(2k-1)}(\ell+1/2)^{\big(1-\frac{k}{N+1}\big)}}.
\end{eqnarray}
Inequality (\ref{blork4}) yields
\begin{eqnarray}\label{blork5}
& & \bigg(e^{-\frac{d}{2(\ell+1/2)^\frac{N+2}{N+1}}\frac{A^2}{1-A^2(\ell+1/2))^\frac{-1}{N+1}}}-1\bigg)e^\frac{\|c\|_{2(N+1)}^{2(N+1)}}{(N+1)(2N+1)}\\
& &\quad\leq \frac{Q_{d,\ell} (t)e^{-\sum_{k=1}^N \frac{1}{k(2k-1)}\frac{\|t\|_{2k}^{2k}}{(\ell+1/2)^{2k-1}}}}{\mathrm{SINC}(\pi t)}-e^\frac{\|c\|_{2(N+1)}^{2(N+1)}}{(N+1)(2N+1)}\nonumber\\
& &\quad\leq e^\frac{d A^{2(N+1)}}{(N+1)(2N+1)} \Big(e^{O\bigg( \frac{1}{(\ell+1/2)^\frac{1}{N+1}}\bigg)}-1 \Big).\nonumber
\end{eqnarray}
The left most side of inequality (\ref{blork5}) is of the order $O((\ell+1/2)^{-\frac{N+2}{N+1}})$, and the right most side of inequality (\ref{blork5}) is of the order $O((\ell+1/2)^{-\frac{1}{N+1}})$, where the ``big O'' constants are independent of $c\in [-A,A]^d$.  The lemma follows readily.
\end{proof}

\begin{proof}[Proof of Lemma \ref{techlem2}]
Equivalently, we need to show $$\lim_{\ell \rightarrow \infty} \sup_{c \in [-A,A]^d} \Big| \Big(e^\frac{\| c\|_{2(N+1)}^{2(N+1)}}{(N+1)(2N+1)}  -1 \Big)f\big(c (\ell+1/2)^\frac{2N+1}{2N+2}\big) \Big| = 0.$$  Suppose the contrary.  Let $c_\ell \in [-A,A]^d$ be a value that maximizes the $\ell$-th term in the above limit.  There exists $(\ell_k)_{k \in \mathbb{N}}$, and $\epsilon>0$ such that for all $k \in \mathbb{N}$,
\begin{eqnarray}
\epsilon  &\leq&  \sup_{c \in [-A,A]^d} \Big| \Big(e^\frac{\| c\|_{2(N+1)}^{2(N+1)}}{(N+1)(2N+1)}  -1 \Big)f\big(c (\ell_k+1/2)^\frac{2N+1}{2N+2}\big) \Big|\nonumber\\ &\leq& \Big(e^\frac{d A^{2(N+1)}}{(N+1)(2N+1)}-1\Big)\big| f\big(c_{\ell_k}(\ell_k+1/2)^\frac{2N+1}{2N+2}\big) \big|,\nonumber
\end{eqnarray}
so that the sequence $\big( f\big(c_{\ell_k} (\ell_k+1/2)^\frac{2N+1}{2N+2}\big) \big)_{k \in \mathbb{N}}$ is bounded away from $0$.  This implies there exists $\delta >0$ such that $\big\| c_{\ell_k} (\ell_k+1/2)^\frac{2N+1}{2N+2} \big\|_{2(N+1)} \leq \delta$ for $k \in \mathbb{N}$, that is, $\| c_{\ell_k} \|_{2(N+1)} \leq \delta (\ell_k+1/2)^{-\frac{2N+1}{2N+2}}$.  This forces
\begin{eqnarray}
\epsilon &\leq&  \sup_{c \in [-A,A]^d} \Big| \Big(e^\frac{\| c\|_{2(N+1)}^{2(N+1)}}{(N+1)(2N+1)}  -1 \Big)f\big(c (\ell_k+1/2)^\frac{2N+1}{2N+2}\big) \Big|\nonumber\\  &\leq & \Big(e^\frac{\delta^{2(N+1)}}{(\ell_k+1/2)^{2N+1}(N+1)(2N+1)} -1 \Big)\| f \|_\infty.\nonumber
\end{eqnarray}
The last term in the above inequality has limit $0$ as $\ell \rightarrow \infty$, which is a contradiction.
\end{proof}

\noindent Now we can prove Theorem \ref{simplethm}.\\

\begin{proof}[Proof of Theorem \ref{simplethm}]
If $f \in PW_{[-\pi,\pi]^d}$, Theorem \ref{maintheorem} states that $$f(t) = \frac{\Psi_\ell(t)}{Q_{d,\ell} (t)}\mathrm{SINC(\pi t)} +\xi_\ell (t)$$ where $\xi_\ell \rightarrow 0$ on $\mathbb{R}^d$ both in $L_2$ and $L_\infty$ senses.  By Lemma \ref{techlem1}, we have
\begin{eqnarray}\label{firstpart}
& &\sup_{t \in (E_{\ell,N,A})^d} \Big| \Psi_\ell(t)  e^{-\sum_{k=1}^{N} \frac{1}{k(2k-1)}\frac{\| t \|_{2k}^{2k}}{(\ell+1/2)^{2k-1}}} - e^\frac{\| t\|_{2(N+1)}^{2(N+1)}}{(\ell+1/2)^{2N+1}(N+1)(2N+1)} \frac{\Psi_\ell(t)}{Q_{d,\ell} (t)}\mathrm{SINC(\pi t)} \Big|\\
& &\quad\quad\leq M (\ell+1/2)^{-\frac{1}{N+1}}\sup_{t \in (E_{\ell,N,A})^d}(|f(t)| - |\xi_\ell(t)|),\nonumber
\end{eqnarray}
the right side of which has zero limit.  Also,
\begin{eqnarray}\label{secondpart}
& &\sup_{t \in (E_{\ell,N,A})^d}\quad\Big| \Big(e^\frac{\| t\|_{2(N+1)}^{2(N+1)}}{(\ell+1/2)^{2N+1}(N+1)(2N+1)}  -1 \Big) \frac{\Psi_\ell(t)}{Q_{d,\ell} (t)}\mathrm{SINC(\pi t)} \Big|\\
& &\quad \leq \sup_{t \in (E_{\ell,N,A})^d} \Big| \Big(e^\frac{\| t\|_{2(N+1)}^{2(N+1)}}{(\ell+1/2)^{2N+1}(N+1)(2N+1)}  -1 \Big)f(t) \Big|+\nonumber\\
& & \quad\quad\Big(e^\frac{d A^{2(N+1)}}{(N+1)(2N+1)} -1 \Big) \sup_{t \in (E_{\ell,N,A})^d} |\xi_\ell(t) |,\nonumber
\end{eqnarray}
whose right hand side, by Lemma \ref{techlem2}, also has zero limit.  Combining inequalities (\ref{firstpart}) and (\ref{secondpart}), we obtain
\begin{equation}
\lim_{\ell \rightarrow \infty} \Big\|  \Psi_\ell(t)  e^{-\sum_{k=1}^{N} \frac{1}{k(2k-1)}\frac{\| t \|_{2k}^{2k}}{(\ell+1/2)^{2k-1}}} - \frac{\Psi_\ell(t)}{Q_{d,\ell} (t)}\mathrm{SINC(\pi t)} \Big\|_{L_\infty((E_{\ell,N,A})^d)}=0.\nonumber
\end{equation}
Equation (\ref{simplecor1}) follows by a final application of Theorem \ref{maintheorem}.\\

\noindent Now we prove equation (\ref{simplecor2}).  Lemma \ref{techlem1} and Theorem \ref{maintheorem} imply
\begin{eqnarray}\label{firstpartl2}
 & &\Big\| \Psi_\ell(t)  e^{-\sum_{k=1}^{N} \frac{1}{k(2k-1)}\frac{\| t \|_{2k}^{2k}}{(\ell+1/2)^{2k-1}}} - e^\frac{\| t\|_{2(N+1)}^{2(N+1)}}{(\ell+1/2)^{2N+1}(N+1)(2N+1)} \frac{\Psi_\ell(t)}{Q_{d,\ell} (t)}\mathrm{SINC(\pi t)} \Big\|_{L_2((E_{\ell,N,A})^d)}\\
& &\quad\leq M (\ell+1/2)^{-\frac{1}{N+1}} \|f + \xi_\ell\|_{L_2((E_{\ell,N,A})^d)},\nonumber
\end{eqnarray}
the right hand side of which has zero limit.  Also,
\begin{eqnarray}\label{secondpartl2}
& &\Big\| \Big(e^\frac{\| t\|_{2(N+1)}^{2(N+1)}}{(\ell+1/2)^{2N+1}(N+1)(2N+1)}  -1 \Big) \frac{\Psi_\ell(t)}{Q_{d,\ell} (t)}\mathrm{SINC(\pi t)} \Big\|_{L_2((E_{\ell,N,A})^d)}\\
& &\quad\leq \Big\| \Big(e^\frac{\| t\|_{2(N+1)}^{2(N+1)}}{(\ell+1/2)^{2N+1}(N+1)(2N+1)}  -1 \Big)f(t) \Big\|_{L_2((E_{\ell,N,A})^d)}+\nonumber\\
& & \quad\quad\Big\| \Big(e^\frac{\| t\|_{2(N+1)}^{2(N+1)}}{(\ell+1/2)^{2N+1}(N+1)(2N+1)}  -1 \Big)\xi_\ell(t) \Big\|_{L_2((E_{\ell,N,A})^d)}.\nonumber
\end{eqnarray}
The second term in the right hand side of inequality (\ref{secondpartl2}) is bounded from above by $$\Big(e^\frac{d A^{2(N+1)}}{(N+1)(2N+1)} -1 \Big) \|\xi_\ell\|_{L_2((E_{\ell,N,A})^d)},$$ which has zero limit.  The integrand of the first term in the right hand side of inequality (\ref{secondpartl2}) (as a function over $\mathbb{R}^d$), converges uniformly to zero by Lemma \ref{techlem2}, and is bounded from above by $$\Big(e^\frac{d A^{2(N+1)}}{(N+1)(2N+1)} -1 \Big) |f(t)|^2 \in L_1(\mathbb{R}^d),$$ so this term has zero limit by the Dominated Convergence Theorem. Combining equations (\ref{firstpartl2}) and (\ref{secondpartl2}) yields
\begin{equation}
\lim_{\ell \rightarrow \infty} \Big\|  \Psi_\ell(t)  e^{-\sum_{k=1}^{N} \frac{1}{k(2k-1)}\frac{\| t \|_{2k}^{2k}}{(\ell+1/2)^{2k-1}}} - \frac{\Psi_\ell(t)}{Q_{d,\ell} (t)}\mathrm{SINC(\pi t)} \Big\|_{L_2((E_{\ell,N,A})^d)}=0.\nonumber
\end{equation}
Equation (\ref{simplecor2}) follows by a final application of Theorem \ref{maintheorem}.
\end{proof}

\noindent The optimal growth for any $(E_{\ell,N,k})_\ell$ such that Theorem \ref{simplethm} holds is not known, but an upper bound for the rate is established in the Appendix.

\section{Examples of uniformly invertible exponential Riesz Bases}\label{S:6}

\noindent  Theorems \ref{maintheorem} and \ref{simplethm} both require uniform invertibility of $(f_n)_{n \in \mathbb{Z}^d}$.  Fortunately, there are significant classes of exponential Riesz bases  $(f_n)_{n \in \mathbb{Z}^d}$ which have this property.  Consider the  exponential Riesz bases described given by Theorem \ref{kadeclike} (Corollary 6.1 in \cite{BB}) and Theorem \ref{sunzhou} (Theorem 1.3 in \cite{SZ}).

\begin{thm}\label{kadeclike}
Let $(t_k)_{k \in \mathbb{Z}^d} \subset \mathbb{R}^d$ such that
\begin{equation}
\sup_{n \in \mathbb{Z}^d} \Arrowvert n - t_n \Arrowvert_\infty = L < \frac{\ln(2)}{\pi d},\nonumber
\end{equation}
then the sequence $(f_k)_{k \in \mathbb{Z}^d}$ defined by $f_k (x) = \frac{1}{(2\pi)^{d/2}}e^{i \langle x , t_k \rangle}$ is a Riesz basis for $L_2([-\pi,\pi]^d)$.
\end{thm}

\begin{thm}\label{sunzhou}
For $d \geq 1$, define $$D_d(x) := \Big(1-\cos \pi x+\sin \pi x  +\frac{\sin \pi x}{\pi x} \Big)^d -\Big( \frac{\sin \pi x}{\pi x} \Big)^d,$$ and let $x_d$ be the unique number such that $0 < x_d \leq 1/4$ and $D_d(x_d)=1$.  Let $(t_k)_{k \in \mathbb{Z}^d} \subset \mathbb{R}^d$ such that  
\begin{equation}
\sup_{n \in \mathbb{Z}^d} \| n - t_n \|_\infty< x_d,\nonumber
\end{equation}
then the sequence $(f_k)_{k \in \mathbb{Z}^d}$  defined by $f_k (x) = \frac{1}{(2\pi)^{d/2}}e^{i \langle x , t_k \rangle}$ is a Riesz basis for $L_2([-\pi,\pi]^d)$.
\end{thm}

\noindent It is worth noting that when $d=1$, Theorem \ref{sunzhou} reduces to the classical Kadec's 1/4 Theorem, first proven in \cite{Kad}. A proof of Kadec's 1/4 Theorem can also be found in \cite[Theorem 14, page 36]{Y}.    The proofs of the above theorems show that the map $L e_n = f_n$ satisfies $\| I- L  \| = \delta < 1$, from which we see that $L$ is invertible.  Uniform invertibility is readily verified, or can be seen as a consequence of the following more general proposition.

\begin{prop}\label{normpert}
Let $L: H \rightarrow H$ be a uniformly invertible operator with respect to $(P_\ell)_{\ell \in \mathbb{N}}$, where $$\limsup_{\ell \rightarrow \infty}\| (P_\ell L P_\ell)^{-1} \| = M <\infty.$$  If $A$ is an operator such that
\begin{equation}\label{neumann}
\| L- A \| =\frac{\gamma}{M}
\end{equation}
for some $\gamma <1$, then there exists $N \in \mathbb{N}$ such that $A$ is uniformly invertible with respect to $(P_\ell)_{\ell \geq N}$.
\end{prop}
\begin{proof}[Proof of Proposition \ref{normpert}]
Using uniform invertiblity of $L$, and noting that $(L_\ell^*)^{-1} -(L^*)^{-1} = (L_\ell^*)^{-1}(L^* - L_\ell^*)(L^*)^{-1}$ and $\lim_{\ell \rightarrow \infty} L_\ell^* x=L^* x$ for all $x \in H$, we see that $ \lim_{\ell \rightarrow \infty} (L_\ell^*)^{-1}x = (L^*)^{-1}x$ for all $x \in H$.  Equation (\ref{ident}) implies that $\lim_{\ell \rightarrow \infty} (L_\ell^*)^{-1} x - (P_\ell L^* P_\ell)^{-1} x =0$ for all $x \in H$. Together we have $$\lim_{\ell \rightarrow \infty} (P_\ell L^* P_\ell)^{-1} x  = (L^*)^{-1}x ,\quad x\in H.$$ General principles imply
\begin{equation}
\| L^{-1} \| \leq \liminf_{\ell \rightarrow \infty} \| (P_\ell L P_\ell)^{-1} \| \leq M.\nonumber
\end{equation}
This yields $\| L - A \| \leq \frac{\gamma}{\| L^{-1}\|}$, so that $A$ is invertible by usual Neumann series manipulation.
Equation (\ref{neumann}) yields that $\| P_\ell L P_\ell- P_\ell A P_\ell \| \leq \frac{\gamma}{M}$, where the norm is now the standard  matrix norm on the set of matrices of  dimension $\mathrm{dim}(\mathrm{ran}{P_\ell})$.  This implies $$ \| P_\ell - (P_\ell L P_\ell)^{-1} (P_\ell A P_\ell)\| \leq \frac{\gamma}{M} \| (P_\ell L P_\ell)^{-1} \|,\quad \ell \in \mathbb{N}.$$  Choose $N$ large enough so that $\| (P_\ell L P_\ell)^{-1} \| \leq \frac{\gamma+1}{2\gamma} M $ when $\ell \geq N$.  This yields $$ \| P_\ell - (P_\ell L P_\ell)^{-1} (P_\ell A P_\ell)\| \leq \frac{\gamma +1}{2}, \quad \ell \geq N.$$
Standard manipulation shows that $P_\ell A P_\ell$ is invertible for $\ell \geq N$, and $$ \sup_{\ell \geq N}\| (P_\ell A P_\ell)^{-1} \| \leq \frac{\gamma+1}{\gamma(1-\gamma)}M.$$
\end{proof}

\noindent Note that in the previous proof, if $M$ is redefined to be $\sup_{\ell \in \mathbb{N}} \| (P_\ell L P_\ell)^{-1} \|$, then $A$ is uniformly invertible with respect to $(P_\ell)_{\ell \in \mathbb{N}}$.\\

\noindent The following proposition shows that  compact perturbations (of arbitrary norm), of a uniformly invertible operator also gives a uniformly invertible operator.

\begin{prop}\label{compactpertthm}
Let $L:H \rightarrow H$ be uniformly invertible with respect to $(P_\ell)_{\ell \in \mathbb{N}}$.  If $\Delta: H \rightarrow H$ is compact such that $\tilde{L} = L + \Delta$ is an onto isomorphism, then there exists $N$ such that $\tilde{L}$ is uniformly invertible with respect to $(P_\ell)_{\ell \geq N}$.
\end{prop}

\begin{proof}[Proof of Proposition \ref{compactpertthm}]
From the definition of $L_\ell$, we have $I = (I-P_\ell)L_\ell^{-1}+L P_\ell L_\ell^{-1}$, so that $$L^{-1}(P_\ell-I)L_\ell^{-1} = P_\ell L_\ell^{-1} -L^{-1}.$$  This implies 
\begin{equation}\label{thisone}
(L_\ell^*)^{-1}P_\ell -(L^*)^{-1} = (L_\ell^*)^{-1}(P_\ell -I)(L^*)^{-1}.
\end{equation}
As $\ell \rightarrow \infty$, the right hand side of equation (\ref{thisone}) has 0 limit pointwise.  Combined with the compactness of $\Delta$, we obtain $$\lim_{\ell \rightarrow \infty} (L_\ell^*)^{-1}P_\ell \Delta = (L^*)^{-1} \Delta$$ where limit is in the operator norm topology.  This yields
\begin{equation}\label{thatone}
\lim_{\ell \rightarrow \infty} I+\Delta P_\ell L_\ell^{-1} = I+\Delta L^{-1} = (L+\Delta)L^{-1},
\end{equation}
where the limit is also in the operator norm topology.  The right had side of equation (\ref{thatone}) is invertible, so there exists $N$ such that $\ell \geq N$ implies $(I+\Delta P_\ell L_\ell^{-1})^{-1}$ exists, and 
\begin{equation}\label{newsup}
\sup_{\ell \geq N} \| (I+\Delta P_\ell L_\ell^{-1})^{-1} \| < \infty.
\end{equation}
From the definition of $\tilde{L}_\ell$, we obtain
\begin{eqnarray}
\tilde{L}_\ell = L_\ell +\Delta P_\ell = (I+\Delta P_\ell L_\ell^{-1})L_\ell.\nonumber
\end{eqnarray}
When $\ell \geq N$, we have $$\tilde{L}_\ell^{-1} = L_\ell^{-1} (I+\Delta P_\ell L_\ell^{-1})^{-1},$$ and equation (\ref{newsup}) implies
\begin{equation}
\sup_{\ell \geq N} \|\tilde{L}_\ell^{-1} \| \leq \sup_{\ell \geq N} \| L_\ell^{-1} \| \sup_{\ell \geq N} \| (I+\Delta P_\ell L_\ell^{-1})^{-1} \| < \infty.
\end{equation}
This completes the proof.
\end{proof}

%\noindent We now apply Proposition \ref{compactpertthm} to uniformly invertible exponential Riesz bases for $L_2([-\pi,\pi]^d)$.\\

\noindent The following lemma holds:

\begin{lem}\label{mainlemma}
Choose $(t_k)_{k \in \mathbb{N}} \subset \mathbb{R}^d$ such that $(h_k)_{k \in \mathbb{N}} :=  \big(\frac{1}{(2\pi)^{d/2}}e^{i\langle (\cdot) , t_k \rangle}\big)_{k \in \mathbb{N}}$ satisfies $$\Big\| \sum_{k=1}^n a_k h_k \Big\|_{L_2 [-\pi,\pi]^d} \leq B \Big(\sum_{k=1}^n |a_k|^2 \Big)^{1/2}, \quad \mathrm{for \ all} \quad (c_k)_{k=1}^n \subset \mathbb{C}.$$  If $(\tau_k)_{k \in \mathbb{N}} \subset \mathbb{R}^d$, and $(f_k)_{k \in \mathbb{N}} := \big(\frac{1}{(2\pi)^{d/2}} e^{i \langle (\cdot) , \tau_k \rangle}\big)_{k \in \mathbb{N}}$, then for all $r,s\geq 1$ and any finite sequence $(a_k)_k$, we have
\begin{equation}%\label{norm}
\Bigg\Arrowvert \sum_{k=r}^s a_k ( h_k - f_k ) \Bigg\Arrowvert_{L_2 [-\pi,\pi]^d} \\
{} \leq B\Big( e^{\pi d \big({\sup \atop {r \leq k \leq s}} \Arrowvert \tau_k-t_k \Arrowvert_\infty\big)}-1 \Big) \Big(\sum_{k=r}^s |a_k|^2\Big)^\frac{1}{2}.\nonumber
\end{equation}
\end{lem}

\noindent This lemma is a slight generalization of Lemma 5.3, found in \cite{BB} using simple estimates.  Lemma \ref{mainlemma} is proven similarly.  A consequence of Lemma \ref{mainlemma} is the following corollary.

\begin{cor}\label{compactcor}
Let $(t_k)_{k \in \mathbb{N}} \subset \mathbb{R}^d$ such that $(f_k)_{k \in \mathbb{N}}$ (the usual exponential sequence defined in terms of $(t_n)_n$) is a Riesz basis for $L_2([-\pi,\pi]^d)$.  Let $(\tau_k)_{k \in \mathbb{N}} \subset \mathbb{R}^d$, and define $(g_k)_{k \in \mathbb{N}}$ by $g_k(x) = \frac{1}{(2\pi)^{d/2}}e^{i \langle x , \tau_k \rangle}$.  Let $(b_k)_{k \in \mathbb{Z}}$ be an orthonormal basis for $L_2([-\pi,\pi]^d)$.  If  $$\lim_{k \rightarrow \infty} \| t_k -\tau_k \|_\infty=0,$$ then the operator defined by $b_k \mapsto f_k - g_k$ is compact.
\end{cor}

\noindent The proof of Corollary \ref{compactcor} is similar to the proof of Corollary 5.5 in \cite{BB}, so it is omitted.

\begin{cor}\label{multidimpert}
Let $(t_k)_{k \in \mathbb{N}} \subset \mathbb{R}^d$. Let $(f_k)_{k \in \mathbb{N}}$ be a  Riesz basis for $L_2([-\pi,\pi]^d)$ which is uniformly invertible with respect to the projections $(P_\ell)_{\ell \in \mathbb{N}}$ defined at the beginning of section 4.  If $(\tau_k)_{k \in \mathbb{N}} \subset \mathbb{R}^d$, and $(g_k)_{k \in \mathbb{N}}$ are as in Corollary (\ref{compactcor}), and additionally, $(g_k)_{k \in \mathbb{N}}$ is a Riesz basis for $L_2([-\pi,\pi]^d)$, then $(g_k)_{k \in \mathbb{N}}$ is uniformly invertible with respect to a sequence of projections $(P_\ell)_{\ell \geq N}$ for some $N >0$.
\end{cor}

\begin{proof}[Proof of Corollary \ref{multidimpert}]
Apply Corollaries \ref{compactpertthm} and \ref{compactcor}.
\end{proof}

\noindent This corollary relates to Theorem \ref{maintheorem} in the following way.  Let $(g_k)_{k \in \mathbb{N}}$ is as in the preceding theorem.  Usage of Corollary \ref{compactpertthm} in the proof of Corollary \ref{multidimpert} does not ensure that low order polynomial interpolants will exist; however, they will existence for sufficiently large $\ell$.  Simple examples show that in Corollary \ref{multidimpert}, the additional assumption that $(g_k)_{k \in \mathbb{N}}$ is a Riesz basis for $L_2([-\pi,\pi]^d)$ cannot be dropped when $d \geq 2$.  Example:  The standard exponential orthonormal basis $(e_n)_{n \in \mathbb{Z}^d}$ is of course uniformly invertible, but the set $$\Big(\frac{1}{(2\pi)^{d/2}}e^{i\langle (\cdot) , (1,1/2,0,\cdots,0) \rangle}\Big)\cup(e_n)_{n \neq 0}$$ is not a Riesz basis.  However, this condition can be dropped when $d=1.$  This follows from Corollary \ref{compactcor} and the following theorem.\\

\begin{thm}\label{pakshinlike}
Let $(t_n)_{n \in \mathbb{Z}}\ \subset \mathbb{R}$ be a sequence such that $(f_n)_{n \in \mathbb{Z}}$ (defined as before) is a Riesz basis for $L_2[-\pi,\pi]$.  If $(\tau_n)_{n \in \mathbb{Z}} \subset \mathbb{R}$ is a sequence of distinct points such that $$\lim_{|n| \rightarrow \infty} |t_n-\tau_n|=0,$$ then $(g_n)_{n \in \mathbb{Z}}$ (defined as before) is a Riesz basis for $L_2[-\pi,\pi].$
\end{thm}

\noindent The proof of Theorem \ref{pakshinlike} relies on Lemma \ref{rbreplace} below, which originally appears as Lemma 3.1 in \cite{PS}.  The proof of Lemma \ref{rbreplace} found in \cite{PS} itself relies on a citation, so for the sake of completeness it is presented here with a self-contained proof.

\begin{lem}\label{rbreplace}
Let $(f_n)_{n \in \mathbb{Z}}$ be an exponential Riesz basis for $L_2[-\pi,\pi]$.  If $(g_n)_{|n| \leq \ell}$ is a sequence of complex exponentials such that
\begin{equation}
(g_n)_{|n| \leq \ell} \cup (f_n)_{|n|>{\ell_0}}\nonumber
\end{equation}
consists of distinct functions, then this set is a Riesz basis for $L_2[-\pi,\pi]$.
\end{lem}

\begin{proof}[Proof of Lemma \ref{rbreplace}]
If we can prove the case when $\ell = 0$, the general result follows inductively.  Let $f_n(\cdot) = \frac{1}{\sqrt{2\pi}}e^{i \langle \cdot , t_n \rangle}$ for $n \neq 0$, and $g_0(\cdot) =  \frac{1}{\sqrt{2\pi}}e^{i \langle \cdot , \tau_0 \rangle}$ where $\tau_0 \in \mathbb{R}$ and $\tau_0 \neq t_n$ for $n \neq 0$.  Let $(e_n)_{n \in \mathbb{Z}}$ be an orthonormal basis for $L_2[-\pi,\pi]$.  Lemma \ref{secondone} shows that $(g_0) \cup (f_n)_{n \neq 0}$ is a Riesz basis if and only if the map defined by $$ e_0 \mapsto g_0, \quad e_k \mapsto f_k, \ \mathrm{for} \ k \neq 0 $$ is one to one.  This is readily seen to be equivalent to $\langle g_0, f_1^* \rangle \neq 0$, or by passing to the Fourier transform, to $G_0 (\tau_0)\neq 0$, (recall that $G_0 = \mathcal{F}(f_0^*)$).  If we can show that the only zeros of $G_0$ in $\mathbb{R}$ are $(t_n)_{n \neq 0}$, we are done.\\

\noindent Suppose there exists $\lambda \in \mathbb{R}$, $\lambda\notin (t_n)_{n \neq 0}$.  Such that $G_0 (\lambda) = 0$ with multiplicity $m$.  Define the entire function $$ H(t) = \frac{(t_0 -\lambda)^m}{(t-\lambda)^m}G_0 (t).$$  Note that $H|_\mathbb{R} \in L_2(\mathbb{R})$, and $H$ is of exponential type $\pi$, so $H \in PW_{[-\pi,\pi]}$ by the Theorem \ref{PaleyWiener}.  The expansion $$H(t) = \sum_{n \in \mathbb{Z}} H(t_n)G_n(t),$$ combined with $H(t_n) = \delta_{n,0}$, shows that $H(t) = G_0(t)$ for all $t \in \mathbb{R}$, an immediate contradiction.  We conclude that $G_0 (\lambda) \neq 0$.
\end{proof}

\begin{proof}[Proof of Theorem \ref{pakshinlike}]
\noindent Define $L e_n = f_n$ and $\tilde{L}e_n = g_n$.  By Corollary \ref{compactcor}, $\tilde{L}$ is bounded linear and $\tilde{L} = L+\Delta$ for some compact operator $\Delta$.  Define the operator \begin{displaymath}
   R_\ell e_n = \left\{
     \begin{array}{lr}
       f_n,  &  |n| \leq \ell\\
       g_n & |n| > \ell
     \end{array}
   \right..
\end{displaymath}
Rewritten, we have
\begin{equation}
R_\ell = L P_\ell +(L+\Delta)(I-P_\ell) = L+\Delta(I-P_\ell).\nonumber
\end{equation}
Compactness of $\Delta$ implies that $\lim_{\ell \rightarrow \infty} R_\ell = L$ in the operator norm topology.  We conclude that $R_{\ell_0}$ is an onto isomorphism for some $\ell_0$ sufficently large; that is, the set 
\begin{equation}\label{dumb}
(f_n)_{|n| \leq \ell_0} \cup (g_n)_{|n|>{\ell_0}}
\end{equation}
is a Riesz basis for $L_2[-\pi,\pi]$.  If we apply Lemma \ref{rbreplace}, by replacing $(f_n)_{|n| \leq \ell_0}$ with $(g_n)_{|n| \leq \ell_0}$ in expression (\ref{dumb}), we have that $(g_n)_{n \in \mathbb{Z}}$ is a Riesz basis for $L_2[-\pi,\pi].$
\end{proof}

\section{Appendix: Comments regarding the optimality of Theorem \ref{simplethm}}\label{S:5}

\noindent In the statement of Theorem \ref{simplethm}, it is not apparent whether or not $(E_{\ell,N,k})_\ell$ can be replaced with a more rapidly growing sequence of intervals; however, Proposition \ref{sinctheorem} shows that if $f(t) = \mathrm{SINC}(\pi t)$, equations (\ref{simplecor1}) and (\ref{simplecor2}) can hold for a sequence of intervals $(E_{\ell,N})_\ell$ which grow faster than $(E_{\ell,N,A})_\ell$.  Propositions \ref{l2bound} and \ref{linfbound} show that growth bounds of the intervals in Proposition \ref{sinctheorem} are optimal for the conclusion of said proposition to hold.  Thus, the bounds in Proposition \ref{sinctheorem} provide upper bounds for the growth of \textit{any} sequence $(E_{\ell,N,A})_\ell$ such that either equation (\ref{simplecor1}) or equation (\ref{simplecor2}) hold for general multivariate bandlimited functions.

\begin{prop}\label{sinctheorem}
Define 
\begin{eqnarray}
C_{\ell,N} & = & \Big(\frac{1}{4}(2N+1)^2 (\ell+1/2)^{2N+1} \log(\ell+1/2)\Big)^\frac{1}{2(N+1)}, \ \mathrm{and}\nonumber\\
D_{\ell,N} & = & \Big(\frac{1}{2}(2N+1)^2 (\ell+1/2)^{2N+1} \log(\ell+1/2)\Big)^\frac{1}{2(N+1)},\nonumber
\end{eqnarray}
then
\begin{equation}\label{sinctheorem1}
\lim_{\ell \rightarrow \infty} \Big\| \mathrm{SINC}(\pi t) - I_{\mathrm{SINC}\pi(\cdot), \ell}(t)  \Big\|_{L_2([-C_{\ell,N},C_{\ell,N}]^d)}=0,
\end{equation}
and
\begin{equation}\label{sinctheorem2}
\lim_{\ell \rightarrow \infty} \Big\| \mathrm{SINC}(\pi t) - I_{\mathrm{SINC}\pi(\cdot), \ell}(t)  \Big\|_{L_\infty([-D_{\ell,N},D_{\ell,N}]^d)}=0.
\end{equation}
\end{prop}

\noindent The proof of equation (\ref{sinctheorem1}) requires the following two propositions.

\begin{prop}\label{auxprop1}
\begin{equation}\label{firstbit}
\lim_{\ell \rightarrow \infty} \Big\|\Big( e^{ \frac{\| t \|_{2(N+1)}^{2(N+1)}}{(\ell+1/2)^{2N+1}(N+1)(2N+1)}}-1\Big) \mathrm{SINC}(\pi t) \Big\|_{L_2([-C_{\ell,N},C_{\ell,N}]^d)} = 0.
\end{equation}
\end{prop}
\begin{proof}[Proof of Proposition \ref{auxprop1}]
Let $t = \alpha C_{\ell,N}$ where $\alpha \in [-1,1]^d$. Noting that $$e^{ \frac{\| t \|_{2(N+1)}^{2(N+1)}}{(\ell+1/2)^{2N+1}(N+1)(2N+1)}} = \Big( \ell+\frac{1}{2} \Big)^{\frac{2N+1}{4(N+1)}\|\alpha \|_{2(N+1)}^{2(N+1)}},$$ the quantity in equation (\ref{firstbit}) becomes
\begin{eqnarray}
& &\bigg(\int_{[-C_{\ell,N},C_{\ell,N}]^d} \bigg|\Big(\Big( \ell+\frac{1}{2} \Big)^{\frac{2N+1}{4(N+1)}\|\alpha \|_{2(N+1)}^{2(N+1)}}-1\Big) \mathrm{SINC}(\pi t)\bigg|^2 dt \bigg)^{1/2}\nonumber\\
& &\quad\leq\frac{1}{C_{\ell,N}^{d/2}}\bigg( \int_{[-1,1]^d} \bigg|\Big( \ell+\frac{1}{2} \Big)^{\frac{2N+1}{4(N+1)}\|\alpha \|_{2(N+1)}^{2(N+1)}}-1\bigg|^2 d \alpha \bigg)^{1/2} \nonumber\\
& &\quad\leq \frac{2^{d/2} \Big( \ell+\frac{1}{2} \Big)^{d\frac{2N+1}{4(N+1)}} +2^{d/2}}{\big(\frac{1}{4}(2N+1)^2\big)^\frac{d}{4(N+1)}\big(\log(\ell+1/2) \big)^\frac{d}{4(N+1)}(\ell+1/2)^{d\frac{2N+1}{4(N+1)}}}.\nonumber
\end{eqnarray}
The last term in the above inequality has limit $0$ as $\ell \rightarrow \infty$.  This proves the proposition.
\end{proof}

\begin{prop}\label{auxprop2}
\begin{eqnarray}
\lim_{\ell \rightarrow \infty} \Big\| Q_{d,\ell}(t) e^{-\!\sum_{k=1}^N \frac{1}{k(2k-1)}\frac{\| t \|_{2k}^{2k}}{(\ell+1/2)^{2k-1}}}\!\!-\!e^{ \frac{\| t \|_{2(N+1)}^{2(N+1)}}{(\ell+1/2)^{2N+1}(N+1)(2N+1)}}\mathrm{SINC}(\pi t)  \Big\|_{L_2([-C_{\ell,N},C_{\ell,N}]^d)}\!=\!0.\nonumber
\end{eqnarray}
\end{prop}

\begin{proof}[Proof of Proposition \ref{auxprop2}]
If $t \in \mathbb{R}^d$ and $\|t \|_\infty <\ell+1/2$, then equation (\ref{blork3}) implies
\begin{eqnarray}\label{needsone}
& & \bigg(e^{\sum_{k=N+1}^\infty \frac{1}{k(2k-1)}\frac{\|t\|_{2k}^{2k}}{(\ell+1/2)^{2k-1}}} \bigg)\prod_{k=1}^d e^{\bigg(-\frac{1}{2(\ell+1/2)}\frac{\big( \frac{t}{\ell+1/2} \big)^2}{1-\big(\frac{t}{\ell+1/2} \big)^2}\bigg)}\\
& &\quad \leq\frac{Q_{d,\ell} (t)e^{-\sum_{k=1}^N \frac{1}{k(2k-1)}\frac{\|t\|_{2k}^{2k}}{(\ell+1/2)^{2k-1}}}}{\mathrm{SINC}(\pi t)} \leq  e^{\sum_{k=N+1}^\infty \frac{1}{k(2k-1)}\frac{\|t\|_{2k}^{2k}}{(\ell+1/2)^{2k-1}}}.\nonumber
\end{eqnarray}
Let $t \in [-C_{\ell,N},C_{\ell,N}]^d$ where $t= \alpha C_{\ell,N}$, $\alpha \in [-1,1]$.  Consider the right hand side of inequality (\ref{needsone}) for such $t$.
\begin{eqnarray}
& & e^{\sum_{k=N+1}^\infty \frac{1}{k(2k-1)}\frac{\|t\|_{2k}^{2k}}{(\ell+1/2)^{2k-1}}} \leq \Big( \ell+\frac{1}{2} \Big)^{\frac{2N+1}{4(N+1)}\|\alpha \|_{2(N+1)}^{2(N+1)}} e^{(\ell+1/2) O \Big( \big\| \frac{t}{\ell+1/2} \big\|_{2(N+2)}^{2(N+2)} \Big)}\\
& &\quad\leq \Big( \ell+\frac{1}{2} \Big)^{\frac{2N+1}{4(N+1)}\|\alpha \|_{2(N+1)}^{2(N+1)}} e^{M (\ell+1/2)^{-\frac{1}{N+1}} (\log(\ell+1/2) )^\frac{N+2}{N+1} \| \alpha \|_{2(N+2)}^{2(N+2)} }.\nonumber
\end{eqnarray}
for some constant $M$. Noting that $$ \frac{t^2}{(\ell+1/2)^3} = \frac{\| \alpha \|_2^2 \big(\frac{1}{4}(2N+1)^2\big)^\frac{1}{N+1}\big( \log(\ell+1/2)\big)^\frac{1}{N+1}}{(\ell+1/2)^\frac{N+2}{N+1}},$$ we can bound the left hand side of inequality (\ref{needsone}) from below as follows:
\begin{eqnarray}\label{blahblah}
& & e^{\Big(-m  \frac{\| \alpha \|_2^2 \big( \log(\ell+1/2)\big)^\frac{1}{N+1}}{(\ell+1/2)^\frac{N+2}{N+1}} \Big)} \Big( \ell+\frac{1}{2} \Big)^{\frac{2N+1}{4(N+1)}\|\alpha \|_{2(N+1)}^{2(N+1)}}\\
& &\quad\leq \bigg(e^{\sum_{k=N+1}^\infty \frac{1}{k(2k-1)}\frac{\|t\|_{2k}^{2k}}{(\ell+1/2)^{2k-1}}} \bigg)\prod_{k=1}^d e^{\bigg(-\frac{1}{2(\ell+1/2)}\frac{\big( \frac{t}{\ell+1/2} \big)^2}{1-\big(\frac{t}{\ell+1/2} \big)^2}\bigg)},\nonumber
\end{eqnarray}
where $m>0$ is chosen independently of $\ell$.  Relations (\ref{needsone}) through (\ref{blahblah}) imply
\begin{eqnarray}
& & \bigg( e^{\Big(-m  \frac{(\log(\ell+1/2))^\frac{1}{N+1}\| \alpha \|_2^2}{(\ell+1/2)^\frac{N+2}{N+1}} \Big)}-1\bigg) \Big( \ell+\frac{1}{2} \Big)^{\frac{2N+1}{4(N+1)}\|\alpha \|_{2(N+1)}^{2(N+1)}} |\mathrm{SINC}(\pi t)| \nonumber\\
& &\quad\leq \bigg| Q_{d,\ell}(t) e^{-\sum_{k=1}^N \frac{1}{k(2k-1)}\frac{\| t \|_{2k}^{2k}}{(\ell+1/2)^{2k-1}}}-e^{ \frac{\| t \|_{2(N+1)}^{2(N+1)}}{(\ell+1/2)^{2N+1}(N+1)(2N+1)}}\mathrm{SINC}(\pi t)  \bigg| \nonumber\\
& &\quad\leq \bigg( e^{\Big(M \frac{ (\log(\ell+1/2) )^\frac{N+2}{N+1} \| \alpha \|_{2(N+2)}^{2(N+2)}}{(\ell+1/2)^\frac{1}{N+1}}\Big)}-1 \bigg) \Big( \ell+\frac{1}{2} \Big)^{\frac{2N+1}{4(N+1)}\|\alpha \|_{2(N+1)}^{2(N+1)}}|\mathrm{SINC}(\pi t)|.\nonumber
\end{eqnarray}
Further simplification implies (for appropriate constants $C$, $C'$, and $C''$) that
\begin{eqnarray}
& & \Big\| Q_{d,\ell}(t) e^{-\sum_{k=1}^N \frac{1}{k(2k-1)}\frac{\| t \|_{2k}^{2k}}{(\ell+1/2)^{2k-1}}}-e^{ \frac{\| t \|_{2(N+1)}^{2(N+1)}}{(\ell+1/2)^{2N+1}(N+1)(2N+1)}}\mathrm{SINC}(\pi t)  \Big\|_{L_2([-C_{\ell,N},C_{\ell,N}]^d)} \nonumber\\
& &\quad\leq C\frac{(\log(\ell+1/2))^\frac{N+2}{N+1}}{(\ell+1/2)^\frac{1}{N+1}} \Bigg(\int_{[-C_{\ell,N},C_{\ell,N}]^d} \Bigg| \Big( \ell+\frac{1}{2} \Big)^{\frac{2N+1}{4(N+1)}\|\alpha \|_{2(N+1)}^{2(N+1)}} \| \alpha \|_2^2 \mathrm{SINC}(\pi t)  \Bigg|^2  dt \Bigg)^{1/2}\nonumber\\
& &\quad =C'\frac{(\log(\ell+1/2))^\frac{N+2}{N+1}}{(\ell+1/2)^\frac{1}{N+1}} \Bigg(\int_{[-1,1]^d} \frac{\Bigg| \Big( \ell+\frac{1}{2} \Big)^{\frac{2N+1}{4(N+1)}\|\alpha \|_{2(N+1)}^{2(N+1)}} \| \alpha \|_2^2 \mathrm{SINC}(\pi t)  \Bigg|^2}{(\log(\ell+1/2))^\frac{d}{2(N+1)}\big(\ell+\frac{1}{2} \big)^{\frac{2N+1}{2(N+1)}d}}  d\alpha \Bigg)^{1/2}\nonumber\\
& &\quad\leq  C''\frac{(\log(\ell+1/2))^\frac{N+2}{N+1}}{(\ell+1/2)^\frac{1}{N+1} (\log(\ell+1/2))^\frac{d}{4(N+1)}},\nonumber
\end{eqnarray}
after the change in variable $t = \alpha C_{\ell,N}$ and simple estimates.  This proves the proposition.
\end{proof}

\begin{proof}[Proof of equation (\ref{sinctheorem1})]
Equation (\ref{sinctheorem1}) follows immediately from Propositions \ref{auxprop1} and \ref{auxprop2}.
\end{proof}

\noindent The proof of equation (\ref{sinctheorem2}) requires the following two propositions.

\begin{prop}\label{auxprop3}
\begin{equation}\label{thirdbit}
\lim_{\ell \rightarrow \infty} \Big\|\Big( e^{ \frac{ t^{2(N+1)}}{(\ell+1/2)^{2N+1}(N+1)(2N+1)}}-1\Big) \mathrm{sinc}(\pi t) \Big\|_{L_\infty [-D_{\ell,N},D_{\ell,N}]} = 0.
\end{equation}
\end{prop}

\begin{proof}[Proof of Proposition \ref{auxprop3}]
Let $t \in [-D_{\ell,N},D_{\ell,N}]$, then $t = \alpha D_{\ell,N}$ for $\alpha \in [-1,1]$. Simplification shows that equation (\ref{thirdbit}) holds if
\begin{equation}\label{sam}
\lim_{\ell \rightarrow \infty} \sup_{\alpha \in [0,1]} \bigg| \frac{(\ell+1/2)^{\alpha^{2(N+1)}\frac{2N+1}{2(N+1)}}-1}{\alpha \big(\log(\ell+1/2)\big)^\frac{1}{2(N+1)} (\ell+1/2)^\frac{2N+1}{2(N+1)}} \bigg| = 0.
\end{equation}
Note that for large $\ell$,
\begin{equation}\label{sam2}
\sup_{\alpha \in [1/2,1]} \Bigg| \frac{(\ell+1/2)^{\alpha^{2(N+1)}\frac{2N+1}{2(N+1)}}-1}{\alpha \big(\log(\ell+1/2)\big)^\frac{1}{2(N+1)} (\ell+1/2)^\frac{2N+1}{2(N+1)}} \Bigg| \leq \frac{2}{\big(\log(\ell+1/2)\big)^\frac{1}{2(N+1)}}.
\end{equation}
\noindent Let $0  < \alpha \leq 1/2$.  The Mean Value Theorem implies
\begin{equation}
\bigg| \frac{(\ell+1/2)^{\alpha^{2(N+1)}\frac{2N+1}{2(N+1)}}-1}{\alpha} \bigg| \leq (2N+1)(\ell+1/2)^{\alpha^{2(N+1)}\frac{2N+1}{2(N+1)}} \alpha^{2N+1} \log(\ell+1/2).
\end{equation}
This yields
\begin{eqnarray}
\sup_{\alpha \in [0,1/2]} \bigg| \frac{(\ell+1/2)^{\alpha^{2(N+1)}\frac{2N+1}{2(N+1)}}-1}{\alpha \big(\log(\ell+1/2)\big)^\frac{1}{2(N+1)} (\ell+1/2)^\frac{2N+1}{2(N+1)}} \bigg|\leq M \frac{\big(\log(\ell+1/2)\big)^\frac{2N+1}{2(N+1)}}{(\ell+1/2)^{\frac{2N+1}{2(N+1)}(1-(1/2)^{2(N+1)})}}\nonumber
\end{eqnarray}
for some constant $M$. Combined with inequality (\ref{sam2}), we have equation (\ref{sam}), which proves the proposition.
\end{proof}

\begin{prop}\label{blip6}
\begin{equation}\label{fourthbit}
\lim_{\ell \rightarrow \infty} \Big\| Q_{1,\ell}(t) e^{-\sum_{k=1}^N \frac{1}{k(2k-1)}\frac{t^{2k}}{(\ell+1/2)^{2k-1}}}-e^{ \frac{t^{2(N+1)}}{(\ell+1/2)^{2N+1}(N+1)(2N+1)}}\mathrm{sinc}(\pi t)  \Big\|_{L_\infty [-D_{\ell,N},D_{\ell,N}]}=0.\nonumber
\end{equation}
\end{prop}

\begin{proof}[Proof of Proposition \ref{blip6}]
Let $t \in [-C_{\ell,N},C_{\ell,N}]$ where $t=\alpha C_{\ell,N}$, $\alpha \in [-1,1]$.  Proceeding in the same manner as in the proof of Proposition \ref{auxprop2}, we see (for appropriate constants $C$ and $C'$) that
\begin{eqnarray}
& & \bigg| Q_{1,\ell}(t) e^{-\sum_{k=1}^N \frac{1}{k(2k-1)}\frac{t^{2k}}{(\ell+1/2)^{2k-1}}}-e^{ \frac{t^{2(N+1)}}{(\ell+1/2)^{2N+1}(N+1)(2N+1)}}\mathrm{sinc}(\pi t)  \bigg|_{L_\infty([-C_{\ell,N},C_{\ell,N}])}\nonumber\\
& &\quad\leq \frac{C (\ell+1/2)^{\alpha^{2(N+1)}\frac{2N+1}{2(N+1)}}\alpha^2 (\log(\ell+1/2))^\frac{N+2}{N+1}|\sin(\pi t)|}{\alpha (\ell+1/2)^\frac{1}{N+1} (\log(\ell+1/2))^\frac{1}{2(N+1)}(\ell+1/2)^\frac{2N+1}{2(N+1)}}\nonumber\\
& &\quad\leq \frac{C' (\log(\ell+1/2))^\frac{2N+3}{2(N+1)}}{(\ell+1/2)^\frac{1}{N+1}}.\nonumber
\end{eqnarray}
This proves the proposition.
\end{proof}

\begin{proof}[Proof of equation (\ref{sinctheorem2})]
\noindent The previous two propositions prove equation when $d=1$.  The multidimensional case follows inductively.
\end{proof}

\begin{prop}\label{l2bound}
Let $N \geq 0$.  If $(M_{\ell,N})_\ell$ is a sequence of positive numbers such that (\ref{sinctheorem1}) holds when $(C_{\ell,N})_\ell$ is replaced by $(M_{\ell,N})_\ell$, then
\begin{equation}\label{l2boundeq}
\limsup_{\ell \rightarrow \infty} \frac{M_{\ell,N}}{C_{\ell,N}} \leq 1.
\end{equation}
\end{prop}

\noindent The proof of Proposition \ref{l2bound} requires the following simple estimate:

\begin{prop}\label{blork6}
Let $a>1/2$, $\epsilon>0$, $0 < \omega <1$, then $$\int_a^{(1+\epsilon)a} \frac{\sin^2 \pi x}{x^{1+\omega}}dx >\frac{\epsilon}{2 a^\omega (1+\epsilon)^\omega}-\frac{a}{2(a-1/2)^{1+\omega}}.$$
\end{prop}
\begin{proof}[Proof of Proposition \ref{blork6}]
Let $b=(1+\epsilon)a$.  We have
$$\int_a^b \frac{\sin^2 \pi x}{x^{1+\omega}}dx + \int_a^b \frac{\cos^2 \pi x}{x^{1+\omega}}dx = \frac{1}{\omega}\Big(\frac{1}{a^\omega}-\frac{1}{b^\omega} \Big)$$
and
$$\int_a^b \frac{\cos^2 \pi x}{x^{1+\omega}}dx = \int_{a-1/2}^{b-1/2} \frac{\sin^2 \pi x}{(x+1/2)^{1+\omega}}dx < \int_{a-1/2}^{b-1/2} \frac{\sin^2 \pi x}{x^{1+\omega}}dx.$$  This yields $$2 \int_a^b \frac{\sin^2 \pi x}{x^{1+\omega}}dx - \int_{b-1/2}^b \frac{\sin^2 \pi x}{x^{1+\omega}}dx+ \int_{a-1/2}^a \frac{\sin^2 \pi x}{x^{1+\omega}}dx > \frac{1}{\omega}\Big(\frac{1}{a^\omega}-\frac{1}{b^\omega} \Big),$$ so that $$\int_a^b \frac{\sin^2 \pi x}{x^{1+\omega}}dx >\frac{1}{2\omega}\Big(\frac{1}{a^\omega}-\frac{1}{b^\omega} \Big) -\frac{1}{2(a-1/2)^{1+\omega}}.$$ Noting that $$\frac{1}{2\omega}\Big(\frac{1}{a^\omega}-\frac{1}{b^\omega} \Big) =\frac{\epsilon}{2\omega a^\omega (1+\epsilon)^\omega} \frac{(1+\epsilon)^\omega -1}{\epsilon}>\frac{\epsilon}{2a^\omega (1+\epsilon)^\omega}$$ proves the proposition.
\end{proof}

\begin{proof}[Proof of Proposition \ref{l2bound}]  Fix $N\geq0$, and define $c=\frac{2N+1}{2N+4}+\delta/2$ where $0<\delta$ is small enough so that $c<1/2$.  Define $$A_\ell = (c(N+1)(2N+1)\log(\ell+1/2))^\frac{1}{2(N+1)}(\ell+1/2)^\frac{-1}{2(N+1)}$$ and $$\epsilon_\ell = (\ell+1/2)^{1-2c} A_\ell.$$  Note that $\lim_{\ell \rightarrow \infty} \epsilon_\ell = 0.$  Let $t \in [A_\ell (\ell+1/2), (1+\epsilon_\ell)A_\ell(\ell+1/2)],$ then $t=\alpha(\ell +1/2)$ for some $\alpha \in [A_\ell,(1+\epsilon_\ell)A_\ell]$. For large $\ell$, note that inequality (\ref{blork3}) implies
\begin{eqnarray}
\frac{1}{2\pi}e^\frac{(\ell+1/2)\alpha^{2(N+1)}}{(N+1)(2N+1)}\frac{|\sin \pi\alpha(\ell+1/2)|}{\alpha (\ell+1/2)} \leq \Big| Q_{1,\ell} (t) e^{-\sum_{k=1}^N \frac{1}{k(2k-1)}\frac{t^{2k}}{(\ell+1/2)^{2k-1}}} \Big|.
\end{eqnarray}
Moving to the multivariate case, if $t \in [A_\ell (\ell+1/2), (1+\epsilon_\ell)A_\ell(\ell+1/2)]^d$, then $t = \alpha (\ell+1/2)$ for some $\alpha \in [A_\ell, (1+\epsilon_\ell)A_\ell]^d$.  This yields
%\begin{eqnarray}
%\prod_{i=1}^d \frac{1}{2\pi}e^\frac{(\ell+1/2)\alpha_i^{2(N+1)}}{(N+1)(2N+1)}\frac{|\sin \pi\alpha_i(\ell+1/2)|}{\alpha_i (\ell+1/2)} \leq \Big| Q_\ell^d (t) e^{-\sum_{k=1}^N \frac{1}{k(2k-1)}\frac{\|t\|_{2k}^{2k}}{(\ell+1/2)^{2k-1}}} \Big|
%\end{eqnarray}
%which gives
\begin{eqnarray}
\prod_{i=1}^d \frac{1}{2\pi\alpha_i^c} \frac{|\sin \pi\alpha_i(\ell+1/2)|}{(\alpha_i(\ell+1/2))^{1-c}} \leq  \Big| Q_{d,\ell} (t) e^{-\sum_{k=1}^N \frac{1}{k(2k-1)}\frac{\|t\|_{2k}^{2k}}{(\ell+1/2)^{2k-1}}} \Big|.\nonumber
\end{eqnarray}
For sufficiently large $\ell$, we can conclude
\begin{eqnarray}
& &\bigg[\frac{1}{9 \pi^2 A_\ell^{2c}} \int_{A_\ell(\ell+1/2)}^{(1+\epsilon_\ell)A_\ell(\ell+1/2)} \frac{\sin^2 \pi x}{x^{2-2c}}dx\bigg]^d\nonumber\\
& &\quad\leq \int_{[A_\ell(\ell+1/2),(1+\epsilon_\ell)A_\ell(\ell+1/2)]^d}  \Big| Q_{d,\ell} (t) e^{-\sum_{k=1}^N \frac{1}{k(2k-1)}\frac{\|t\|_{2k}^{2k}}{(\ell+1/2)^{2k-1}}} \Big|^2 dt.\nonumber
\end{eqnarray}

\noindent Applying Proposition \ref{blork6} for $a= A_\ell (\ell+1/2)$, $\epsilon = \epsilon_\ell,$ and $\omega = 1-2c$, and using the definition of $\epsilon_\ell$, we obtain
\begin{eqnarray}
& &\bigg[\frac{1}{9\pi^2}\Big[\frac{1}{2(1+\epsilon_\ell)^{1-2c}} - \frac{1}{2 A_\ell^{2c}(A_\ell(\ell+1/2)-1)^{2-2c}} \Big]\bigg]^d\nonumber\\
& &\quad\leq \int_{[A_\ell(\ell+1/2),(1+\epsilon_\ell)A_\ell(\ell+1/2)]^d}  \Big| Q_{d,\ell} (t) e^{-\sum_{k=1}^N \frac{1}{k(2k-1)}\frac{\|t\|_{2k}^{2k}}{(\ell+1/2)^{2k-1}}} \Big|^2 dt.\nonumber
\end{eqnarray}
The first term in the brackets in the previous equation has limit $1/2$, while the second term has limit $0$. We conclude there exists a constant $\beta >0$ such that
\begin{equation}\label{blork7}
\beta \leq \int_{[A_\ell(\ell+1/2),(1+\epsilon_\ell)A_\ell(\ell+1/2)]^d}  \Big| Q_{d,\ell} (t) e^{-\sum_{k=1}^N \frac{1}{k(2k-1)}\frac{\|t\|_{2k}^{2k}}{(\ell+1/2)^{2k-1}}} \Big|^2 dt, \quad \ell >0.
\end{equation}
If $M_{\ell,N} \geq (\ell+1/2)(1+\epsilon_\ell)A_\ell$ for infinitely many $\ell$, there exists a subsequence $(\ell_k)_{k \in \mathbb{N}}$ such that (in particular),
\begin{eqnarray}
\lim_{\ell_k \rightarrow \infty}\! \Big\| \mathrm{SINC}(\pi t) - Q_{d,\ell_k}(t) e^{-\sum_{k=1}^{N} \frac{1}{k(2k-1)}\frac{\| t \|_{2k}^{2k}}{(\ell_k+1/2)^{2k-1}}}\Big\|_{L_2([A_{\ell_k} (\ell_k+1/2)), A_{\ell_k} (\ell_k+1/2)(1+\epsilon_{\ell_k})]^d)}\!\!=\!0.\nonumber
\end{eqnarray}
This contradicts inequality (\ref{blork7}).  This yields that for sufficiently large $\ell$, 
\begin{eqnarray}
 M_{\ell,N} &<& (\ell+1/2)(1+\epsilon_\ell)A_\ell\nonumber\\
& = & (1+\epsilon_\ell)\Big(\Big(\frac{2N+1}{4N+4}+\delta/2 \Big)(N+1)(2N+1)(\ell+1/2)^{2N+1}\log(\ell+1/2)\Big)^\frac{1}{2(N+1)}.\nonumber
\end{eqnarray}
Note that since $\epsilon_\ell \rightarrow 0$, for large $\ell$, the quantity $(1+\epsilon_\ell)\Big(\frac{2N+1}{4N+4}+\delta/2 \Big)^\frac{1}{2(N+1)}$ is less than, (and bounded away from) the quantity $\Big(\frac{2N+1}{4N+4}+\delta \Big)^\frac{1}{2(N+1)}$.  We conclude that for any $\delta>0$, there exists $\ell_{N,\delta}$ such that $$ \sup_{\ell >\ell_{N,\delta}} \frac{M_{\ell,N}}{((N+1)(2N+1)\log(\ell+1/2))^\frac{1}{2(N+1)}(\ell+1/2)^\frac{2N+1}{2(N+1)}} < \Big( \frac{2N+1}{4N+4} +\delta \Big)^\frac{1}{2(N+1)}.$$  Proposition \ref{l2bound} follows.
\end{proof}

\begin{prop}\label{linfbound}
Let $N \geq 0$.  If $(M_{\ell,N})_\ell$ is a sequence of positive numbers such that equation (\ref{sinctheorem2}) holds when $(D_{\ell,N})_\ell$ is replaced by $(M_{\ell,N})_\ell$, then
\begin{equation}
\limsup_{\ell \rightarrow \infty} \frac{M_{\ell,N}}{D_{\ell,N}} \leq 1.
\end{equation}
\end{prop}

\noindent The proof of Proposition \ref{linfbound} requires the following fact:

\begin{prop}\label{sinprop}
Let $0 < \epsilon \leq 1$.  If $A>0$, there exists $t \in [A, A+\epsilon]$ such that $|\sin(\pi t)| \geq |\sin(\pi \epsilon/2)|$.
\end{prop}
\begin{proof}[Proof of Proposition \ref{sinprop}]
The proof is clear from geometric considerations.
\end{proof}

\begin{proof}[Proof of Proposition \ref{linfbound}]
Let $N \geq 0$. Choose $\delta>0$ such that $c := \frac{2N+1}{2N+2}+\delta/2<1$. Define $$A_\ell = (c(N+1)(2N+1)\log(\ell+1/2))^\frac{1}{2(N+1)}(\ell+1/2)^\frac{-1}{2(N+1)}$$ and $$\epsilon_\ell = A_\ell (\ell+1/2)^{1-c}.$$ Note that $\lim_{\ell \rightarrow \infty} \epsilon_\ell = 0$.  Let $t \in [A_\ell (\ell+1/2),A_\ell(\ell+1/2)+\epsilon_\ell].$ Proceeding as before, for sufficiently large $\ell$, we have
\begin{equation}
\frac{1}{2\pi}e^{\Big(\frac{t^{2(N+1)}}{(\ell+1/2)^{2N+1}(N+1)(2N+1))}\Big)} \frac{|\sin(\pi t)|}{t} \leq \Big| Q_{1,\ell} (t) e^{-\sum_{k=1}^N \frac{1}{k(2k-1)}\frac{t^{2k}}{(\ell+1/2)^{2k-1}}}\Big|.\nonumber
\end{equation}
Now for all $t \in [A_\ell (\ell+1/2),A_\ell(\ell+1/2)+\epsilon_\ell]$,\\
\begin{equation}
\frac{1}{2\pi}\frac{(\ell+1/2)^c}{A_\ell(\ell+1/2)+\epsilon_\ell}|\sin(\pi t)| \leq \Big| Q_{1,\ell} (t) e^{-\sum_{k=1}^N \frac{1}{k(2k-1)}\frac{t^{2k}}{(\ell+1/2)^{2k-1}}}\Big|.\nonumber
\end{equation}
In the multivariate case, if $t \in [A_\ell (\ell+1/2),A_\ell(\ell+1/2)+\epsilon_\ell]^d$, we obtain
\begin{equation}
 \frac{1}{(2\pi)^d}\frac{(\ell+1/2)^{cd}}{(A_\ell(\ell+1/2)+\epsilon_\ell)^d}\prod_{i=1}^d |\sin(\pi t_i)| \leq \Big| Q_{d,\ell} (t) e^{-\sum_{k=1}^N \frac{1}{k(2k-1)}\frac{\|t\|_{2k}^{2k}}{(\ell+1/2)^{2k-1}}}\Big|.\nonumber
\end{equation}
For large $\ell$, an application of Proposition \ref{sinprop} yields
\begin{equation}
\frac{1}{(3\pi)^d}\frac{|\sin(\pi \epsilon_\ell/2)|^d}{A_\ell^d(\ell+1/2)^{(1-c)d}} \leq \Big\| Q_{d,\ell} (t) e^{-\sum_{k=1}^N \frac{1}{k(2k-1)}\frac{\|t\|_{2k}^{2k}}{(\ell+1/2)^{2k-1}}} \Big\|_{L_\infty([A_\ell (\ell+1/2),A_\ell(\ell+1/2)+\epsilon_\ell]^d)}.\nonumber
\end{equation}
By the definition of $\epsilon_\ell$, the right hand side of the above equation tends to a positive constant.  The remainder of the proof is almost identical to that of Proposition \ref{l2bound}.
\end{proof}

\noindent The following is trivially deduced from Propositions \ref{l2bound} and \ref{linfbound}:  Fix $N>0$.  If $(E_{\ell,N})_\ell$ is a sequence of intervals such that either equation (\ref{simplecor1}) or equation (\ref{simplecor2}) holds for all $f \in PW_{[-\pi,\pi]^d}$, then $$\max_{x \in (E_{\ell,N})^d} \| x \|_\infty = o\Big(\max_{x \in (E_{\ell,N+1,A})^d} \| x \|_\infty\Big).$$

\end{document}